\documentclass[11pt]{amsart} 
\usepackage{amsmath, amssymb, amsthm,latexsym} 
\usepackage[table]{xcolor} 
\usepackage{booktabs, multirow}    
\usepackage{libertine,url,cite}
\usepackage[foot]{amsaddr} 
\usepackage{fancyhdr} 
\newcommand\textcode[1]{{\footnotesize \rm\texttt{#1}}}

\newcommand\hcancel[2][black]{\setbox0=\hbox{$#2$}%
\rlap{\raisebox{.45\ht0}{\textcolor{#1}{\rule{\wd0}{1pt}}}}#2}

\def\mmgroup{\textcode{mmgroup}}
\def\etal{et~al.\ }

\newcommand{\sg}[3]{\subsection{The group \label{#2} #1}
\begin{theorem}
    \label{thm_#2}
    #3
\end{theorem} 
} 

\newcommand{\sgNEW}[4]{\subsection{The group \label{#2} #1}
#4
  \begin{theorem}
    \label{thm_#2}
    #3
\end{theorem}
} 

\DeclareMathOperator{\im}{Im}

\newcommand{\thm}[2][]{\def\firstargtemp{#1}%
    Theorem\ifx\firstargtemp\empty\else s~\ref{thm_#1} and\fi~\ref{thm_#2}}
\newcommand{\prp}[1]{Proposition~\ref{#1}}
\newcommand{\lem}[1]{Lemma~\ref{#1}}

\newcommand{\mtt}[1]{{\rm #1}}

\newcommand{\CO}{\textup{Co}_1}
\newcommand{\FI}{\textup{Fi}_{24}}
\newcommand{\SZ}{\textup{Suz}}
\newcommand{\HE}{\textup{He}}
\newcommand{\HN}{\textup{HN}}
\newcommand{\Th}{\textup{Th}}
\newcommand{\QStr}{2^{1+24}}
\newcommand{\BB}{\mathbb{B}}
\newcommand{\MM}{\mathbb{M}}
\newcommand{\QQ}{\mathbf{Q}}
\newcommand{\GG}{\mathbf{G}}
\newcommand{\J}[1]{\textup{J}_{#1}}
\newcommand{\MT}[1]{\textup{M}_{#1}}
\newcommand{\dih}[1]{\textup{D}_{#1}}
\newcommand{\sdih}[1]{\textup{SD}_{#1}}
\newcommand{\psl}[2]{{\rm PSL}_{#1}(#2)} 

\newcommand{\pgl}[2]{{\rm PGL}_{#1}(#2)}

\newcommand{\ling}[2]{{\rm GL}_{#1}(#2)}

\newcommand{\lins}[2]{{\rm SL}_{#1}(#2)}
\newcommand{\unt}[2]{{\rm U}_{#1}(#2)}
\newcommand{\orth}[3]{{\rm O}^{#1}_{#2}(#3)} 
\newcommand{\lie}[3]{{\rm #1}_{#2}(#3)}
\newcommand{\alt}[1]{{\rm A}_{#1}}
\newcommand{\sym}[1]{{\rm S}_{#1}}
\newcommand{\symp}[2]{{\rm S}_{#1}(#2)}
\newcommand{\aut}[1]{\textup{Aut}(#1)} 
\newcommand{\cm}[1]{C_\MM(#1)}
\newcommand{\nm}[1]{N_\MM(#1)}
\newcommand{\cma}[1]{C_\MM(\langle #1 \rangle)}
\newcommand{\nma}[1]{N_\MM( \langle #1 \rangle)}

\def\normACA{\left( \alt{5} \times \unt{3}{8}{:}3 \right){:}2}
\def\normAAA{\left( \alt{5} \times \alt{12} \right){:}2}
\def\normLThreeTwo{\left( \psl{3}{2} \times \symp{4}{4}{:}2 \right) \udot 2}
\def\normASeven{\left( \alt{7} \times \left( \alt{5} \times \alt{5} \right){:}2^2 \right){:}2}
\def\normASixCu{\left( \alt{6} \times \alt{6} \times \alt{6} \right).\left( 2 \times \sym{4} \right)}
\def\normASix{\MT{11} \times \alt{6} \udot 2^2}
\def\normSFiveCu{\left( \sym{5} \times \sym{5} \times \sym{5} \right){:}\sym{3}}
\def\normLTwoElevenSq{\left( \psl{2}{11} \times \psl{2}{11} \right){:}4}
\def\normLTwoEleven{\left( \psl{2}{11} \times \MT{12} \right){:}2}
\def\normThirteenBSq{13^2{:}\lins{2}{13}{:}4}
\def\normThirteenB{13^{1+2}{:}\left( 3 \times 4.\sym{4} \right)}
\def\normThirteenA{\left( 13{:}6 \times \psl{3}{3} \right) \udot 2}
\def\normElevenSq{11^2{:}\left( 5 {\times} 2.\alt{5} \right)}
\def\normSevenBSqOut{7^2{:}\lins{2}{7}}
\def\normSevenBSq{7^{2+1+2}{:}\ling{2}{7}}
\def\normSevenASq{\left( 7^2{:} \left( 3 \times 2.\alt{4} \right) \times \psl{2}{7} \right){:}2}
\def\normSevenB{7^{1+4}{:}\left( 3 \times 2.\sym{7} \right)}
\def\normSevenA{\left( 7{:}3 \times \HE \right){:}2}
\def\normFiveBBq{5^4{:}\left( 3 \times 2 \udot \psl{2}{25} \right){:}2}
\def\normFiveBCu{5^{3+3} \udot \left( 2 \times \psl{3}{5} \right)}
\def\normFiveBSq{5^{2+2+4}{:}\left( \sym{3} \times \ling{2}{5} \right)}
\def\normFiveASq{\left( 5^2{:}4 \udot 2^2 \times \unt{3}{5} \right){:}\sym{3}}
\def\normFiveB{5^{1+6}{:}2 \udot \J2{:}4}
\def\normFiveA{\left( \dih{10} \times \HN \right) \udot 2}
\def\normThreeArk{3^8 \udot \orth-{8}{3}.2}
\def\normThreeBCu{3^{3+2+6+6}{:}\left( \psl{3}{3} \times \sdih{16} \right)}
\def\normThreeBSq{3^{2+5+10}{:}\left( \MT{11} \times 2.\sym{4} \right)}
\def\normThreeASq{\left( 3^2{:}2 \times \orth+{8}{3} \right) \udot \sym{4}}
\def\normThreeC{\sym{3} \times \Th}
\def\normThreeB{3^{1+12} \udot 2 \udot \SZ{:}2}
\def\normThreeA{3 \udot \FI}
\def\normTwoArk{2^{10+16} \udot \orth+{10}{2}}
\def\normTwoBQu{2^{5+10+20} \udot \left( \sym{3} \times \psl{5}{2} \right)}
\def\normTwoBCu{2^{3+6+12+18} \udot \left( \psl{3}{2} \times 3.\sym{6} \right)}
\def\normTwoBSq{2^{2+11+22} \udot \left( \MT{24} \times \sym{3} \right)}
\def\normTwoASq{2^2 \udot {}^2 \lie{E}{6}{2} {:} \sym{3}}
\def\normTwoB{\QStr \udot \CO}
\def\normTwoA{2 \udot \BB}

\makeatletter 
\renewcommand\subsubsection{\@startsection{subsubsection}{3}%
  \z@{.5\linespacing\@plus.7\linespacing}{-.5em}%
  {\normalfont\bfseries}} 
\makeatother
 
\usepackage[T1]{fontenc} 

\newcommand\makequotestraight{%
\begingroup\lccode`~=`' 
\lowercase{\endgroup\let~}\textquotesingle
\catcode`'=\active
}

\usepackage[a4paper,right=2.5cm, left=2.5cm, top=2.5cm, bottom=2.5cm, marginpar=1.6cm]{geometry}

\newtheoremstyle{mythm}                   
{6pt}
{6pt}
{\it}
{}
{\bf}
{.}
{.5em}
{}

\newtheoremstyle{mydef}                   
{6pt}
{6pt}
{}
{}
{\bf}
{.}
{.5em}
{}

\newtheoremstyle{myrem}                   
{6pt}
{6pt}
{}
{}
{\bf}
{.}
{.5em}
{}

\theoremstyle{mythm}      
\newtheorem{theorem}{Theorem}[section]
\newtheorem{proposition}[theorem]{Proposition}
\newtheorem{lemma}[theorem]{Lemma}

\theoremstyle{mydef}

\theoremstyle{myrem}
\newtheorem{remark}[theorem]{Remark}
\newtheorem{procedure}[theorem]{Procedure}
\numberwithin{equation}{section}

\newcounter{ithmcount}

\newenvironment{iprf}{\begin{list}{{\rm
	\alph{ithmcount})}}{\usecounter{ithmcount}\labelwidth-5pt
      \leftmargin0pt \topsep3pt \itemsep1pt \parsep2pt}}{\end{list}}

\newenvironment{iprff}{\begin{list}{{\rm
	\alph{ithmcount})}}{\usecounter{ithmcountt}\labelwidth18pt
      \leftmargin18pt \topsep3pt \itemsep1pt \parsep2pt}}{\end{list}}

\newenvironment{ithm}{\begin{list}{{\rm \alph{ithmcount})}}{\usecounter{ithmcount}\labelwidth18pt
      \leftmargin18pt \topsep3pt \itemsep1pt \parsep2pt}}{\end{list}}

\renewcommand{\leq}{\leqslant} 
\renewcommand{\geq}{\geqslant}
\newcommand{\PSL}{{\rm PSL}}

\newcommand{\Aut}{{\rm Aut}}
\newcommand{\GL}{{\rm GL}}

\newcommand{\mt}[1]{{\rm #1}}

\makeatletter
\newcommand{\udot}{\mathpalette\udot@\relax}  
\newcommand{\udot@}[2]{%
  \begingroup
  \sbox\z@{$#1{:}$}%
  \sbox\tw@{$#1{.}$}%
  \raisebox{\dimexpr\ht\z@-\ht\tw@}{$\m@th#1.$}%
  \endgroup
}
\makeatother

\usepackage{fancyvrb}
\fvset{commandchars=\\\{\}}

\makeatletter
\@namedef{subjclassname@1991}{2020 MSC}
\makeatother

\begin{document}

\title{Explicit construction of the maximal subgroups of the Monster}
\subjclass[2000]{} 
\author[H. Dietrich]{Heiko Dietrich}
\author[M. Lee]{Melissa Lee}
\author[A. Pisani]{Anthony Pisani}
\author[T. Popiel]{Tomasz Popiel}
\address[Dietrich, Lee, Pisani, Popiel]{School of Mathematics, Monash University, Clayton VIC 3800, Australia}
\email{\rm heiko.dietrich@monash.edu, melissa.lee@monash.edu, apis0001@student.monash.edu,}
\email{\rm  tomasz.popiel@monash.edu}
\thanks{This work forms part of Pisani's Honours thesis at Monash University. Lee acknowledges the support of an Australian Research Council Discovery Early Career Researcher Award (project number DE230100579).}
\date{\today}

\begin{abstract}
  Seysen's Python package \mmgroup{} provides functionality for fast computations within the sporadic simple group $\MM$, the Monster. The aim of this work is to present an \mmgroup{} database of maximal subgroups of $\MM$: for each conjugacy class $\mathcal{C}$ of maximal subgroups in $\MM$, we construct  explicit group elements in \mmgroup{} and prove that these elements generate a group in~$\mathcal{C}$. Our generators and the computations verifying correctness are available in accompanying code. The maximal subgroups of $\MM$ have been classified in a number of papers spanning several decades; our work constitutes an independent verification of these  constructions. We also correct the claim that $\MM$ has a maximal subgroup $\psl{2}{59}$, and hence identify  a new maximal subgroup $59{:}29$.  
\end{abstract}

\maketitle

\thispagestyle{empty}

\section{Introduction}\label{sec_intro}  

\noindent The classification of the maximal subgroups of the Monster group $\MM$, the largest sporadic simple group, is a result of decades of effort led primarily by Holmes, Norton, and Wilson. In 2017, Wilson \cite{survey} noted that over $15$ papers had been dedicated to this classification, yet a few open cases cases seemed largely resistant to theoretical arguments and posed significant computational difficulties. These cases were recently settled by Dietrich, Lee, and Popiel \cite{dlp} with the help of the newly available Python software package \mmgroup{} developed by Seysen \cite{serpent,fast_monster,mmgroup}. For more background information and a comprehensive discussion, we refer to \cite{dlp} and the references therein. 

While much of the early classification work was theoretical, extensive computations eventually became necessary to construct or rule out the existence of various almost simple maximal subgroups. Several constructions, in fact, rely on calculations using a non-standard computational model developed by Holmes and Wilson \cite{monp}. Although groundbreaking at the time, this model is not publicly available, making it difficult to verify or reproduce the results. Seysen's \mmgroup{} package is a game changer in this regard, providing, for the first time, a publicly available framework for \emph{fast} computations in~$\MM$. In the  course of completing the maximal subgroup classification, the work in \cite{dlp} yielded explicit generators in $\MM$ (in \mmgroup{} format) for the maximal subgroups $\mathrm{PGL}_2(13)$ and $\mathrm{PSU}_3(4)$, along with several auxiliary subgroups of $\MM$. The aim of this paper is to consider the remaining maximal subgroups of $\MM$, as described in the existing literature, and to provide explicit generators for these groups in \mmgroup{} format with computationally reproducible proofs of correctness. Our generators (and computations verifying their correctness) are available in accompanying Python code  \cite{pisgit}. Our explicit calculations in \mmgroup{} confirm the maximal subgroup constructions available in the literature, with the exception of the construction of $\psl{2}{59}$ in \cite{l2_59}; this leads to our first result.

\begin{theorem}\label{thm_psl259}
The Monster has no subgroup $\psl{2}{59}$, but a unique class of subgroups $59{:}29$, which are maximal subgroups.
\end{theorem}

Theorem \ref{thm_psl259} is in conflict with the main result of \cite{l2_59} which claims that $\MM$ \emph{has} a maximal subgroup $\psl{2}{59}$. We prove Theorem \ref{thm_psl259} in Section \ref{sec_psl259} and illustrate in detail how the final conclusion of \cite{l2_59} is in contradiction with our calculations. Exhibiting elements in $\MM$ generating $59{:}29$ presents significant computational challenges and is the subject of ongoing work, see Remark \ref{remget59}.

The main result of this paper can be summarised as follows.

\begin{theorem}\label{main_thm}
For each maximal subgroup $N$ of the Monster $\MM$ listed in Table \ref{maxsgtab}, with the exception of $59{:}29$, the accompanying code \cite{pisgit} provides explicit elements of $\MM$ in \mmgroup{} format that generate a maximal subgroup of $\MM$ isomorphic to $N$.
\end{theorem}

With the exception of the groups $\psl{2}{59}$ and $59{:}29$, which will be discussed in Section \ref{sec_psl259}, the correctness of this table follows from \cite{survey} and \cite{dlp}, see in particular \cite[Theorem~1.7]{dlp}.  Generators for the maximal subgroups $\normTwoA$ and $\normTwoB$ are already provided in \cite{dlp} and in \mmgroup{}. For each remaining maximal subgroup $N\ne 59{:}29$, we prove that the associated elements  provided in the accompanying code \cite{pisgit} generate a maximal subgroup of $\MM$ that is isomorphic to $N$.  Specifically, $p$-local maximal subgroups with $p$ odd are discussed in Section \ref{oddlocal}.   The $2$-local maximal subgroups are discussed in Section \ref{2local}. Section \ref{a5s} provides some necessary details on alternating subgroups $\alt{5}$ in $\MM$. Non-local maximal subgroups are considered in Sections \ref{psgl} and \ref{nonlocal}.

\def\mygray{gray!10}
\def\myblue{blue!10}

\def\mygray{white}
\def\myblue{white}

\renewcommand {\arraystretch} {1.1}
  \begin{table}[t]
\begin{tabular}{l|l|l} \specialrule{\heavyrulewidth}{0pt}{0pt}
\cellcolor{\myblue} $\normTwoA$ & \cellcolor{\mygray} $\normFiveA$ & \cellcolor{\myblue} $\normAAA$ \\
\cellcolor{\myblue} $\normTwoB$ & \cellcolor{\mygray} $\normFiveB$ & \cellcolor{\mygray} $\normASixCu$ \\
\cellcolor{\mygray} $\normTwoASq$ & \cellcolor{\mygray} $\normFiveASq$ & \cellcolor{\mygray} $\normACA$ \\
\cellcolor{\mygray} $\normTwoBSq$ & \cellcolor{\mygray} $\normFiveBSq$ & \cellcolor{\mygray} $\normLThreeTwo$ \\
\cellcolor{\mygray} $\normTwoBCu$ & \cellcolor{\mygray} $\normFiveBCu$ & \cellcolor{\mygray} $\normLTwoEleven$ \\
\cellcolor{\mygray} $\normTwoBQu$ & \cellcolor{\mygray} $\normFiveBBq$ & \cellcolor{\mygray} $\normASeven$ \\
\cellcolor{\mygray} $\normTwoArk$ & \cellcolor{\mygray} $\normSevenA$ & \cellcolor{\mygray} $\normASix$ \\
\cellcolor{\myblue} $\normThreeA$ & \cellcolor{\mygray} $\normSevenB$ & \cellcolor{\mygray} $\normSFiveCu$ \\
\cellcolor{\mygray} $\normThreeB$ & \cellcolor{\mygray} $\normSevenASq$ & \cellcolor{\myblue} $\normLTwoElevenSq$ \\
\cellcolor{\myblue} $\normThreeC$ & \cellcolor{\mygray} $\normSevenBSq$ & \cellcolor{\myblue} $\unt{3}{4}{:}4$ \\
\cellcolor{\mygray} $\normThreeASq$ & \cellcolor{\myblue} $\normSevenBSqOut$ & \cellcolor{\mygray} $\psl{2}{71}$ \\
\cellcolor{\mygray} $\normThreeBSq$ & \cellcolor{\myblue} $\normElevenSq$ & \cellcolor{\mygray} $\hcancel{\psl{2}{59}}$ \\
\cellcolor{\mygray} $\normThreeBCu$ & \cellcolor{\myblue} $\normThirteenA$ & \cellcolor{\mygray} $\psl{2}{41}$ \\
\cellcolor{\mygray} $\normThreeArk$ & \cellcolor{\mygray} $\normThirteenB$ & \cellcolor{\mygray} $\pgl{2}{29}$ \\
$~$ & \cellcolor{\mygray} $\normThirteenBSq$ & \cellcolor{\myblue} $\pgl{2}{19}$ \\
$59{:}29$ & \cellcolor{\mygray} $41{:}40$ & \cellcolor{\myblue} $\pgl{2}{13}$ \\
\specialrule{\heavyrulewidth}{0pt}{0pt} \end{tabular}
\vspace*{1ex}

\caption[The maximal subgroups of $\MM$]{\label{maxsgtab}\small{%
    The maximal subgroups of $\MM$. For each group, we provide generators in \mmgroup{} format; see Section \ref{correction} for comments on  $\pgl{2}{59}$ and ${59{:}29}$. 
}}
    \end{table}  

\subsection{How to read this paper}
\label{sec:how_to_read}

This is a computational paper and the main result of our work is a database of maximal subgroups of $\MM$ in \mmgroup{} format, made available in a well-documented Jupyter Notebook \cite{pisgit}. This paper is concerned with the proof that the database is correct. Our proof relies on theoretical results from the classification of the maximal subgroups, and on explicit computations in \mmgroup{}. We also often refer to data concerning the maximal subgroups of various simple groups listed in the Atlas \cite{atlas,gap_ctbllib}. (For example, we establish that a group $G$ is generated by given elements of orders $n$ and $m$ because, according to the Atlas, there is no maximal subgroup of $G$ that has elements of these orders.) We recommend that  our proofs are  read alongside the information provided in the Notebook, because the latter often contains the calculations proving the claims made in the text.

For each maximal subgroup $N\ne 59{:}29$ of $\MM$ listed in Table \ref{maxsgtab}, there is a corresponding subsection in this paper and in the Jupyter Notebook. We define certain elements of $\MM$ in \mmgroup{} format and prove that the group generated by these elements is maximal in $\MM$ and isomorphic to $N$.  Depending on whether $N$ is $p$-local or not, we employ different strategies that are outlined below.

\noindent {\bf Notation.}
Most of our notation is standard and typically aligns with the conventions established in the Atlas \cite{atlas}. One significant exception is that we use $\PSL_d(q)$ in place of $\mtt{L}_d(q)$, and apply similar conventions for other simple classical groups. Unless specified otherwise, we denote by $2\mt{A}$, $2\mt{B}$, $3\mt{A}$, etc.\ the conjugacy classes of $\MM$, and for a conjugacy class $p$X we write $p$X$^{k}$ for an elementary abelian group of size $p^k$ whose non-identity elements all belong to $p$X; such a group is called $p$X-{\em pure}. We denote the dihedral group of order $n$ by $\dih{n}$, the alternating group of degree $n$ by $\alt{n}$, and the symmetric group of degree $n$ by $\sym{n}$. The symbol $n$ is also used to denote a cyclic group of order $n$. An extension of a group $B$ by a group $A$ is denoted by $A.B$ (or occasionally $AB$), with $A$ being the normal subgroup. To emphasise that an extension is split, we may use the notation $A{:}B$, whereas $A\udot B$ indicates a non-split extension. We call this description of a group the \emph{shape} of the group, and stress that the shape does in general not define the group up to isomorphism. (For example, there are in general many non-isomorphic groups that are extensions of $B$ by $A$, that is,  of shape $A.B$.) We follow the convention that $A.B.C=(A.B).C$. 
An elementary abelian group of order $p^k$ is denoted by $p^k$, where $p$ is a prime number and $k$ is a positive integer, while $p^{k+\ell}$ denotes an extension $p^k.p^\ell$. We often use a subscript to indicate the order of a group element; for instance, $g_5$ may refer to an element of order 5. Finally, we denote by $\GG$ the centraliser $C_{\MM}(z)$ of a distinguished 2B involution $z$.  Extended functionality and faster computation are available in \mmgroup{} for elements of $\GG$; see \cite{fast_monster} for details.


    \section{\label{oddlocal}The Odd-Local Maximal Subgroups of $\MM$}
    
    \noindent
    A subgroup $U$ of $\MM$ is $p$-{\em local} if $U=N_\MM(E)$ for some $p$-subgroup $E$; it is \emph{maximal $p$-local} if $U$ is maximal among all $p$-local subgroups with respect to inclusion. We say that $U$ is \emph{$p$-local maximal} if $U$ is $p$-local and also a maximal subgroup of $\MM$. Incorporating earlier work of Norton, Wilson \cite{odd_local} classified the maximal $p$-local subgroups of $\MM$ for odd $p$. However, questions about maximality of these groups were only resolved in all cases 20 years later
    through computational constructions of containing subgroups of the Monster, see for example \cite{subgroups_A5}. 
     The updated results from \cite{odd_local} pertaining to maximal subgroups of
     $\MM$ are summarised in Table~\ref{max_odd_tab} and Proposition \ref{bigoddlocalprop}.

     \subsection{Preliminary lemmas}
     We start with two preliminary results that are frequently used; the first is an easy observation. 

    \begin{lemma}
    	\label{ext_lemma}
    	Let $G$ be a finite group with $H \le G$ and $g \in G$. If $g \notin H$, then
    	$| \langle H, g \rangle | \ge 2 | H |$. Moreover,
    	if the order of $g$ is a prime power $p^n$ and $g^{(p^{n-1})} \notin H$,
    	then $| \langle H, g \rangle | \ge p^n | H |$.
    \end{lemma}
    
    \begin{proof}
        If $g \notin H$, then $H$ and $gH$ are distinct cosets 
        in $\langle g, H \rangle$, so $| \langle g, H \rangle | \geq
            2 | H |$. If $g$ has prime-power order $p^n$, then 
        $\langle g \rangle \cap H = \langle g^{(p^i)} \rangle$ for some $i$.
        By assumption, $g^{(p^{n-1})} \notin H$, so $\langle g \rangle \cap H = 1$.
        Thus, the cosets $H, gH, g^2 H, \ldots , g^{p^n-1} H$ are
        all distinct in $\langle g, H \rangle$, and therefore  
        $| \langle g, H \rangle | \geq p^n | H |$. 
    \end{proof}
    
    The next technical lemma plays an important role in our proofs to justify that we have constructed the full $p$-cores (i.e.\ maximal normal $p$-subgroups) $O_p(N)$ of the various $p$-local subgroups $N$ of $\MM$.
    
    \begin{lemma}
        \label{odd_ext} Let $G$ be a finite group, $x\in G$ with prime order $p$, and $N \unlhd N_G(\langle x\rangle)$ with $|N|=p^{2k + 3}$ for some $k\geq 0$. Suppose $x\in N$ and there exist $y, \ell, g_1, \ldots ,  g_k, h_1, \ldots , h_k \in N \cap C_G(x)$ of order $p$ and $\sigma \in N_G(\langle x,y\rangle)$ such that $y \notin \langle x \rangle$; all $g_i, h_i$ commute with $y$, whereas $\ell$ does not; all $g_j$ commute modulo $\langle x \rangle$   with all $g_i$, $g_i^{\sigma}$, $h_i$; all $g_j$ commute with $h_i^{\sigma}$ modulo $\langle x\rangle$ when $i < j$,  but not when $i = j$.
        Then the following hold.
        \begin{ithm}
          \item Every $S \subseteq \{ x, y, g_1, \ldots , g_k, h_1, \ldots ,
            h_k, g_1^{\sigma}, \ldots , g_k^{\sigma} \}$ generates
            a $p$-group of order at least $p^{| S |}$.
          \item The group $N$ is generated by $\{x, y, g_1, \ldots , g_k, h_1, \ldots, h_k, \ell \}$.
           \item The group $\langle x, y, g_1, \ldots , g_k, h_1, \ldots , h_k, g_1^{\sigma}, \ldots
            g_k^{\sigma}, h_1^{\sigma}, \ldots , h_k^{\sigma} \rangle$ is a normal $p$-subgroup of $C_G(\langle x,y\rangle)$.
        \end{ithm}
    \end{lemma}
    
    \begin{proof}
     \begin{iprf} 
\item        Let $\phi\colon C_G(x)\to C_G(x)/\langle x\rangle$ be the natural homomorphism. We first show that $S$ generates
        a group (not necessarily a $p$-group) of order
        at least $p^{| S |}$. Suppose, for a contradiction, that $S$ is
        a counterexample of minimal size, that is, $|\langle S\setminus\{u\}\rangle|\geq p^{|S|-1}$ for all $u\in S$;  note that $|S|\geq 1$ since $| \langle \emptyset \rangle | = p^0$. Let $u\in S$ be the element that occurs latest in the list $x, y, g_k, \ldots , g_1, h_1,
            \ldots , h_k, h_1^{\sigma}, \ldots , h_k^{\sigma}$; note the reversed labels for the $g_i$. The following case distinction shows that  $u\notin \langle S\setminus\{u\}\rangle$.
        \begin{iprff} 
            \item[$\bullet$] If $u = x$, then  $S \setminus \{ u \}=\emptyset$, and  so $x\notin \langle S\setminus\{u\}\rangle$.
            \item[$\bullet$] If $u = y$, then  $S\setminus\{u\} \subseteq \{ x \}$, and $y\notin \langle x\rangle$ by assumption.
            \item[$\bullet$] If $u = g_i$ for some $1 \le i \le k$, then
                $S \setminus \{ u \}$ consists of
                at most $x$, $y$, and $g_j$ for various $j > i$. Note that
                $g_m^{\sigma}, h_m^{\sigma}$ lie in $(
                    C_G(x) \cap C_G(y) )^{\sigma} = C_G(\langle x, y \rangle^{\sigma}) =
                    C_G(\langle x, y\rangle)$, as do $g_m, h_m$ for all $1 \le m \le k$.
                Thus
                $\phi ( h_i^{\sigma} )$ commutes
                with  $\phi ( x )$,
                $\phi ( y )$, and 
                $\phi ( g_j )$ for $j > i$, so  $\phi (
                    \langle S \setminus \{ u \} \rangle
                    )$ centralises $\phi (h_i^{\sigma})$. Since $u=g_i$ does not commute with $h_i^\sigma$ modulo $\langle x\rangle$ by assumption, we deduce $u\notin \langle S\setminus\{u\}\rangle$.
            \item[$\bullet$] If $u = h_i$ for some $1 \le i \le k$, then
                $S \setminus \{ u \}$ consists of
                at most $x$, $y$, the $g_m$, and various $h_j$
                for $j < i$. Consider $\phi ( (
                    S \setminus \{ u \}
                )^{\sigma} )$. Each $\phi (
                    g_m^{\sigma}
                )$ and $\phi ( h_j^{\sigma} )$ with $i > j$ commutes with $\phi ( g_i )$ by assumption,
                as do $\phi ( x^{\sigma} ), \phi (
                    y^{\sigma}
                ) \in \phi ( \langle x, y \rangle )$.
                It follows that $\phi ( \langle (
                    S \setminus \{ u \}
                )^{\sigma} \rangle )$ centralises $\phi ( g_i )$. Since $\phi ( u^{\sigma} ) =
                    \phi ( h_i^{\sigma} )$ does not commute
                with $\phi ( g_i )$ by assumption, we deduce $u\notin\langle S\setminus\{u\}\rangle$.
            \item[$\bullet$] If $u = h_i^{\sigma}$ for some $1 \le i \le k$, then
                $S \setminus \{ u \}$ consists of
                at most $x$, $y$, the $g_m$, the $h_m$, and $h_j^{\sigma}$
                for $j < i$. All of $\phi ( x )$,
                $\phi (y)$, $\phi (g_m)$, $\phi (h_m)$, and $\phi (h_j^{\sigma})$
                commute with $\phi (g_i)$ by assumption and,  as before,
                 $\phi ( \langle S \setminus \{ u \} \rangle )$ centralises  $\phi ( g_i )$.
                As before,  $u=h_i^\sigma$ is not contained in  $\langle S\setminus\{u\}\rangle$.
        \end{iprff}
Thus,  $u \notin \langle S \setminus \{ u \} \rangle$, and \lem{ext_lemma} yields $| \langle S \rangle | \ge p |
            \langle S \setminus \{ u \} \rangle
            | \ge p^{1 + | S | - 1 } = p^{| S |}$, which contradicts that $S$ is  counterexample. Thus,  $|\langle S\rangle|\geq p^{|S|}$ for all $S$. It will follow from c) that $\langle S\rangle$ is a $p$-group.
        \item The group  $A=\langle x, y, g_1, \ldots g_k, h_1, \ldots h_k\rangle$ lies in $N \cap C_G(y)$ and has size  $|A|\geq p^{2k+2}$ by a). The element    $\ell \in N \setminus C_G(y)$ has order $p$, so $\langle N \cap C_G(y), \ell \rangle\leq N$ has size at least $p | N \cap C_G(y) | \ge p^{1+2k+2} =
        | N |$ by  \lem{ext_lemma}. This implies that $A = N \cap C_G(y)$  has order $p^{2k+2}$ and $N=\langle A,\ell\rangle$, as claimed.
      \item We continue with the previous notation. All generators of $A=N \cap C_G(y)$ lie in $C_G(x)$, so  $A=N\cap C_G(\langle x,y\rangle)$. Since $N \unlhd N_G (\langle x\rangle)$, we have  $A\unlhd C_G(\langle x,y\rangle)$.
        But then $A^{\sigma} \unlhd C_G(\langle x,y\rangle)^{\sigma} =
            C_G(\langle x, y \rangle^{\sigma}) = C_G(\langle x, y\rangle)$, which proves the first claim of c), that is,
        \[ \langle x, y, g_1, \ldots , g_k, h_1, \ldots , h_k, g_1^{\sigma},
            \ldots g_k^{\sigma}, h_1^{\sigma}, \ldots , h_k^{\sigma} \rangle =
            \langle A, A^{\sigma} \rangle \unlhd C_G(\langle x,y\rangle); \]
        note that the generators $x^{\sigma}, y^{\sigma} \in \langle x, y \rangle$
        are redundant. Moreover, $A,A^\sigma\unlhd \langle A,A^\sigma\rangle$, and so $| \langle
            A, A^{\sigma}
        \rangle | = | AA^{\sigma} | =| A ||
            A^{\sigma}|/| A \cap A^{\sigma} |$ divides
        $| A | | A^{\sigma} | = | A |^2 =
            p^{4k+4}$. Thus, $\langle A, A^{\sigma} \rangle$ is a $p$-group.\qedhere
     \end{iprf}
    \end{proof}

    
    \subsection{The $p$-local subgroups}

For completeness, we recall the classification of the odd-local maximal subgroups of $\MM$, with the correction discussed in Section \ref{correction} (that is, with the addition of $59{:}29$).
    
    \begin{proposition}[\!{\cite{odd_local}}; \cite{survey};
        \cite{subgroups_A5}; \S\ref{correction}]
        \label{3localprop}
        \label{5localprop}
        \label{7localprop}
        \label{bigoddlocalprop}
        The odd-local maximal subgroups of $\MM$ are, up to conjugacy, 
        the normalisers $N_\MM(E)$  listed in Table~\ref{max_odd_tab}.
    \end{proposition}
    
    \def\chrule{\addlinespace[2pt] \cline{1-3}  \addlinespace[2pt]}
    \newlength\citcomw
    \settowidth\citcomw{\cite[Theorems~3, 7.1]{odd_local}; $E$ extends
      $3^7 < 3^{1+12} \le \nm{\textup{3B}}$xxx}
    \begin{table}[t]
        \begin{tabular}{lll}
            \toprule
            $\nm{E}$ & $E$ & Citations and Comments \\
            \midrule
            $\normThreeA$ & 3A & \multirow{4}{\citcomw}{see \cite[Theorem~3]{odd_local};
                for 3A$^2$, see also the table above \cite[Proposition 2.2]{odd_local}} \\
            $\normThreeB$ & 3B & \\
            $\normThreeC$ & 3C & \\
            $\normThreeASq$ & 3A$^2$ & \\ \cmidrule{3-3}
            $\normThreeBSq$ & 3B$^2$ & \multirow{2}{\citcomw}{%
                see \cite[Theorems~3, 6.5]{odd_local}; the structure of $E$ is
                insufficient for the maximality and structure of $\nm{E}$.} \\
            $\normThreeBCu$ & 3B$^3$ & \\ \cmidrule{3-3}
            $\normThreeArk$ & $3^8$ & see \cite[Theorems~3, 7.1]{odd_local}; $E$ extends
                $3^7 < 3^{1+12} \le \nm{\textup{3B}}$ \\ \chrule
            $\normFiveA$ & 5A & \multirow{5}{\citcomw}{see \cite[Theorem~5]{odd_local};
                for 5A$^2$, 5B$^2$ see also the first two tables in \cite[\S9]{odd_local}
            } \\
            $\normFiveB$ & 5B & \\
            $\normFiveASq$ & 5A$^2$ & \\
            $\normFiveBSq$ & 5B$^2$ & \\
            $\normFiveBCu$ & 5B$^3$ & \\ \cmidrule{3-3}
            $\normFiveBBq$ & 5B$^4$ & see \cite[Theorem~5]{odd_local}; $E$ extends
                $5^2 < 5^{1+6} \le \nm{\textup{5B}}$ \\ \chrule
            $\normSevenA$ & 7A & \multirow{3}{\citcomw}{see \cite[Theorem~7]{odd_local};
                for 7A$^2$ see also the table at the start of \cite[\S10]{odd_local}} \\
            $\normSevenB$ & 7B & \\
            $\normSevenASq$ & 7A$^2$ & \\ \cmidrule{3-3}
            $\normSevenBSq$ & 7B$^2$ & see \cite[Theorem~7]{odd_local};
                $E < 7^{1+4} \lhd \nm{\textup{7B}}$ \\ \cmidrule{3-3}
            $\normSevenBSqOut$ & 7B$^2$ & see \cite[Theorem~7]{odd_local};
                $E \not< 7^{1+4} \lhd \nm{\textup{7B}}$ \\ \chrule
            $\normElevenSq$ & $11^2$ & \multirow{6}{\citcomw}[-5pt]{see \cite[\S11]{odd_local};
                note that $\MM$ has unique conjugacy classes of elements of orders
                11, 41, and that the two classes of elements of order 59 are
                inverses } \\ \cmidrule{1-2}
            $\normThirteenA$ & 13A & \\
            $\normThirteenB$ & 13B & \\
            $\normThirteenBSq$ & 13B$^2$ & \\ \cmidrule{1-2}
            $41{:}40$ & 41 & \\ \cmidrule{1-2}
            $59{:}29$ & 59 & \\ 
            \bottomrule
        \end{tabular} 

\vspace*{1ex}
        
        \caption[The odd-local maximal subgroups $N_\MM(E)$ of $\MM$.]{\label{max_odd_tab}\small{The odd-local maximal subgroups $N_\MM(E)$ of $\MM$.}}
    \end{table}
   
    The following subsections consider each of the odd-local maximal subgroups separately and prove that the generators in the accompanying code do indeed generate a maximal subgroup of the correct {shape}. For a more efficient  exposition, we first outline the common steps of these proofs.

    \begin{procedure}\label{provesg}
     Let $N=N_\MM(E)$ be one of the $p$-local maximal subgroups in Table \ref{max_odd_tab}. The accompanying code provides two lists, $L_E$ and $L_N$, of elements in $\MM$ in \mmgroup{} format. Let $\tilde E=\langle L_E\rangle$ and $\tilde N=N_\MM(\tilde E)$. We aim to prove that $E\cong \tilde E$ and  $N\cong \tilde N=\langle L_N\rangle$. We usually proceed as follows, and explain alternative approaches in the relevant proofs.
    \begin{iprf}
    \item[(1)] We first confirm that $\tilde E$ is an elementary abelian group of the correct type, see column ``$E$'' in Table \ref{max_odd_tab}. For this we confirm computationally that $\tilde E$ is elementary abelian of the correct size, and, if required, that all its non-identity elements lie in the correct conjugacy class of $\MM$. We note that \mmgroup{} provides the functionality to compute the values of the unique irreducible $198663$-dimensional complex character $\chi_\MM$ of $\MM$ for elements that lie in the maximal subgroup $\GG=N_{\MM}(2\mt{B})$.  Once the type of $\tilde E$ is established, the structure and maximality of $N_{\MM}(\tilde E)\cong N$ is usually determined by Proposition \ref{bigoddlocalprop}. 
    \item[(2)] A direct calculation in \mmgroup{} shows that $L_N$ normalises $\tilde E$, and  the next step is to show that $\langle L_N\rangle$ contains  $C_\MM(\tilde E)$. This centraliser is usually an extension of the form $A.B$, where the structures of $A$ and $B$ are known. We exhibit words in the elements of $L_N$ whose cosets with respect to $A$ generate a group isomorphic to $C_\MM(\tilde E)/A\cong B$; if $C_\MM(\tilde E)=A{:}B$ is known to be a split extension, then we often exhibit words in the elements of $L_N$ that generate $B$. The last step is to construct, as words in elements of $L_N$, generators of $A$. Here it is useful that $\gcd(|A|,|B|)$ is usually small and that $A$ contains low-index $p$-subgroups. In conclusion, at the end of this step we have shown that  $C_\MM(\tilde E)\leq \langle L_n\rangle \leq \tilde N=N_\MM(\tilde E)$. 
    \item[(3)] Finally, we consider the homomorphism $\phi \colon \langle L_N\rangle \to  \aut{\tilde E}$ induced by conjugation. Step (2) establishes that $\cm{\tilde E} = \ker{\phi}$, so the image of $\phi$ is isomorphic to a subgroup of $N_\MM(\tilde E)/C_\MM(\tilde E)$. A direct calculation allows us to determine the size of the image of $\phi$, and if this equals $|N_\MM(\tilde E)/C_\MM(\tilde E)|$, then $\langle L_N\rangle = \tilde N$ is established.
    \end{iprf}
\end{procedure}
    We use the notation and approach of Procedure \ref{provesg} in each of the following proofs, sometimes with minor modifications. We always denote by $L_E$ and $L_N$  the lists of elements provided in the accompanying code, and always write \[\tilde E=\langle L_E\rangle\quad\text{and}\quad \tilde N=N_\MM(\tilde E).\]
    We then prove that $\tilde N=\langle L_N\rangle$ is the required maximal subgroup. 
    We usually do not comment on obvious verifications, such as checking that the elements in $L_N$ normalise $\tilde E$, etc, but these  verifications are done in the corresponding section of the accompanying Jupyter Notebook.

    \sgNEW{$\normThreeA$}{norm3A}{
      The accompanying code defines $L_E=\{x_3\}$ and $L_N=\{g_3,h_3\}$. The group $N_\MM(\langle L_E\rangle)$ is generated by $L_N$ and  isomorphic to the maximal subgroup $\normThreeA$ of $\MM$.
    }{The next theorem is  adapted from the accompanying code of \cite{dlp}.}

    \begin{proof}The element $x_3\in \GG$ has order $3$ and  $\chi_\MM ( x_3 ) = 782$, so $x_3\in 3\mt{A}$. Therefore,  $N_{\MM}(\tilde E )\cong\normThreeA$ and $C_\MM(\tilde E)\cong \tilde E$.  We confirm  $x_3 = ( g_3 h_3 g_3 h_3^3 g_3 h_3^5 )^{28}\in \langle L_N\rangle$. To ensure that all cosets of $N_\MM(\tilde E)/\tilde E$ have representatives in $\langle L_N\rangle$, we consider the elements  $a = g_3 h_3$ and    $b = ( g_3 h_3 )^5 h_3$ of order $29$ and $70$, respectively. It follows that also $aC_\MM(\tilde E)$ and $bC_\MM(\tilde E)$ have order $29$ and $70$, respectively, and they are elements in $\FI\cong N_\MM(\tilde E)/\tilde E$.  Atlas information shows that $\FI$  has   a unique class of maximal subgroup whose order is divisible by $29$ and $70$, but this subgroup has no elements of order $70$. We deduce that  $aC_\MM(\tilde E)$ and $bC_\MM(\tilde E)$ generate a group isomorphic to $\FI$, and so $N_\MM(\tilde E)=\langle L_N\rangle$.
    \end{proof}


    \sg{$\normThreeB$}{norm3B}{ The accompanying code defines $L_E=\{x_{3b}\}$ and $L_N=\{g_{3b},h_{3b}\}$. The group $N_\MM(\langle L_E\rangle)$ is generated by $L_N$ and  isomorphic to the maximal subgroup $\normThreeB$ of $\MM$.
    }
     
    \begin{proof}
The element $x_{3b}\in \GG$ has order $3$ and
$\chi_\MM ( x_{3b} ) = 53$, so $x_{3b}\in 3\mt{B}$. Thus, $\tilde E$ has the correct type and $N_{\MM}(\tilde E )\cong\normThreeB$ is a maximal subgroup. The elements $g_{3b}$ and $h_{3b}$ normalise $\tilde E$ and have orders $56$ and $66$, respectively, and therefore some of their powers project to elements of orders $28$ and $11$
    in the factor group $\SZ.2$ under the natural projection map $\varphi$. Atlas information shows that $\SZ.2$ has a unique class of maximal subgroups
    of order divisible by both $28$ and $11$, namely $\SZ$, but this subgroup has no element of order $28$. This implies that $\langle L_N\rangle$ contains a complete set of coset representatives of $U = \ker \varphi \cong 3^{1+12}.2$ in $\tilde N\cong \normThreeB$. It remains to show that $U$ lies in $\langle L_N\rangle$. Since $\SZ.2$ contains
    no elements of order $56$ and $\gcd{( 56, 3^{1+12} \cdot 2 )} = 2$, we have $g_{3b}^{28}\in U$. Since $g_{3b}^{28}$ has order $2$, it lies in $U\setminus 3^{1+12}$. Thus, it remains to show that $\langle L_N\rangle$ contains $V=3^{1+12}$. The element $r=g_{3b}^{28} (g_{3b}^{28})^{h_{3b}}$ lies in $V$, and in the accompanying code we express $x_{3b}$ as a word in conjugates of $r$. We also define elements  $y_{3b}$, $\ell$, $g_1,\ldots,g_5$, $h_1,\ldots,h_5$, $\sigma$ as words in elements of $L_N$ and apply Lemma~\ref{odd_ext} to show that these generate $V$. This completes the proof.
    \end{proof}

    \sgNEW{$\normThreeC$}{norm3C}{
 The accompanying code defines $L_E=\{x_{3c}\}$ and $L_N=\{g_{3c},h_{3c}\}$. The group $N_\MM(\langle L_E\rangle)$ is generated by $L_N$ and  isomorphic to the maximal subgroup $\normThreeC$ of $\MM$.
    }{The next theorem is proved in the with elements adapted from \cite{dlp}.}

    \begin{proof} We exhibit an element $c\in\MM$ such that $x_{3c}^c\in \GG$ has order $3$ and $\chi_\MM ( x_{3c}^c ) = -1$, proving that $\tilde E$ has the correct form and $N_{\MM}(\tilde E )\cong \normThreeC$. Let $a = g_{3c}^3$ and $b = h_{3c}^4$. The elements $ab$ and $abab^2 ab^2 abab^2 abababab^2 ab^2 abab$ have orders $19$ and $31$, respectively, and their cosets with respect to $\sym{3}$ generate a subgroup of $\Th$ containing elements of these orders. Atlas information implies that this subgroup is the whole group $\Th$. It remains to show that $\langle L_N\rangle$ contains $\sym{3}$. Since $\tilde E$ is normal in $N_\MM(\tilde E)$, it must lie in $\sym{3}$. Now confirming that   $x_{3c} = g_{3c}^4$ and $h_{3c}^3$ generate $\sym{3}$ completes the proof.    \end{proof}

    \sg{$\normThreeASq$}{norm3A2}{
The accompanying code defines $L_E=\{x_{3},y_3\}$ and $L_N=\{g_{3a2},h_{3a2}\}$. The group $N_\MM(\langle L_E\rangle)$ is generated by $L_N$ and  isomorphic to the maximal subgroup $\normThreeASq$ of $\MM$.
    }
    
    \begin{proof}
    We confirm that $\tilde E$ is $3\mt{A}$-pure of size $3^2$. There is a unique class of such subgroups, see \cite[\S 2]{odd_local}, so $\tilde N\cong \normThreeASq$. Using the auxiliary element   $g = h_{3a2}^2 (g_{3a2} h_{3a2})^2 g_{3a2}^{-1} h_{3a2}^{-1}$, we confirm that 
      $x_{13} = (g_{3a2}^2 h_{3a2})^2 
        g_{3a2}^4 h_{3a2}^{-2}$, $x_2 = ((h_{3a2} g_{3a2})^2
            h_{3a2} g_{3a2}^{-1})^{15}$, and
    $x_{14} = (g x_2 g^{-1} x_2 g)^3$  have orders $13$, $2$, and $14$, respectively, and centralise $\tilde E$. Thus,  they lie in the factor
            $\orth+{8}{3}$ of $C_\MM(\tilde E) \cong 3^2 \times \orth+{8}{3}$. By \cite[p. 140]{atlas}, the only maximal subgroup
    of $\orth+{8}{3}$ with order divisible by both $13$ and $14$ is
    $\orth{}{7}{3}$, in which the centraliser of an element of order $7$ has
    order $14$. Confirming that $x_{14}^2$ (of order $7$) commutes
    with $x_2 (\neq x_{14}^7)$ as well as with $x_{14}$, we conclude that
    $\orth+{8}{3}$ is contained in $\tilde E$.  We then confirm  that
    $x_3 = (h_{3a2} g_{3a2}^3 (h_{3a2} g_{3a2})^2)^8$ and
    $y_3 = (h_{3a2} g_{3a2} h_{3a2} g_{3a2}^3 h_{3a2} g_{3a2})^8$, and since $\orth+{8}{3}$ has a trivial centre, $\tilde E\leq \langle L_N\rangle$, and so $C_\MM(\tilde E)\leq \langle L_N\rangle$. Lastly, we obtain $48=|2.\sym{4}|$ distinct automorphisms induced by the conjugation action of $\langle L_N\rangle$ on $\tilde E$, which proves  that $\tilde N=\langle L_N\rangle$.
    \end{proof}

    \sg{$\normThreeBSq$}{norm3B2}{
The accompanying code defines $L_E=\{x_{3b},y_{3b}\}$ and $L_N=\{g_{3b2},h_{3b2}\}$. The group $N_\MM(\langle L_E\rangle)$ is generated by $L_N$ and  isomorphic to the maximal subgroup $\normThreeBSq$ of $\MM$.
    }
    
    \begin{proof}
      We check that $\tilde E\cong 3^2$ and $\langle L_N\rangle \leq \tilde N$. By \cite[Theorem~5.2]{odd_local}, the normaliser $\tilde N$ is contained
    in a maximal subgroup with structure $\normThreeA$, $\normThreeB$, $\normThreeASq$, $\normThreeBSq$, or
    $\normThreeArk$; since only the fourth of these groups contains an element of order  $|g_{3b2}|=88$, we must have $\tilde N\leq \normThreeBSq$. (We note that \cite[Theorem~5.2]{odd_local} refers to maximal $3$-local subgroups,
    but by \prp{3localprop}, all such subgroups of $\MM$ are maximal.) It remains to prove that $|\tilde N|\geq |\normThreeBSq|= 3^{17} | \MT{11} | \cdot 48$. We first confirm that  $\langle h_{3b2}^6, (g_{3b2}^8 h_{3b2}^3 g_{3b2}^{32})^2 \rangle\cong \MT{11}$ by verifying a presentation for $\MT{11}$.  We then establish elements in $\langle L_N\rangle$ that satisfy the assumptions of \lem{odd_ext} (with $x=x_{3b}$ and $y=y_{3b}$) and so generate a $3$-group of size at least $3^{17}$ that is normal in $C_\MM(\tilde E)$.  Together with the simple group $\MT{11} \le C_\MM(\tilde E)$
    exhibited above, we have generated  a subgroup of $C_\MM(\tilde E) \cap \langle L_N\rangle$
    of order at least $3^{17} | \MT{11} |$. We conclude the proof by enumerating the $48$ distinct automorphisms of $\tilde E$ induced by conjugation by $\langle L_N\rangle$. 
    \end{proof}

    \sg{$\normThreeBCu$}{norm3B3}{
The accompanying code defines $L_E=\{x_{3b},y_{3b},z_{3b}\}$ and $L_N=\{g_{3b3},h_{3b3}\}$. The group $N_\MM(\langle L_E\rangle)$ is generated by $L_N$ and  isomorphic to the maximal subgroup $\normThreeBCu$.
    }
    
    \begin{proof} 
     \thm{norm3B2} shows that 
    $\langle x_{3b}, y_{3b} \rangle \cong 3^2 < 3^{1+12} \unlhd
     \nma{x_{3b}} \cong \normThreeB$. Setting $c=g_{3b3} h_{3b3}^3 g_{3b3} h_{3b3} g_{3b3}$,  we verify that $c\in \cm{x_{3b}}$, that $z_{3b} = y_{3b}^c$ and $z_{3b}$ commutes with $x_{3b}$ and $y_{3b}$, so $\tilde E\cong 3^3$ is a subgroup of $3^{1+12}$ containing the centre. Now \cite[Theorem~6.5]{odd_local} implies that $\tilde N$  belongs to a maximal subgroup of $\MM$ of type $\normThreeB$, $\normThreeBSq$, or $\normThreeBCu$. We confirm $g_{3b3},h_{3b3}\in \tilde N$ and $|h_{3b3}|=104$, which forces $\tilde N\leq \normThreeBCu$. The elements  $a = h_{3b2}^6$ and $b = (g_{3b2}^8 h_{3b2}^3 g_{3b2}^{32})^2$ from the proof of \thm{norm3B2} generate a subgroup isomorphic to $\MT{11}$, which is contained in $N_\MM(\langle x_{3b},y_{3b}\rangle)=\normThreeBSq$. We check that  $u=(aba (b^2 a)^2)^2$ and $v=((ba)^3 (b^2 a)^2 ba)^2$ generate a group of size $9$ that centralises $\tilde E$. In the accompanying code, we establish $15$ elements (as words in elements of $L_N$) that commute with all elements of $L_E$ and lie in the subgroup $3^{2+5+10}\leq N_\MM(\langle x_{3b},y_{3b}\rangle)$. An application of \lem{odd_ext} proves that these elements generate a $3$-subgroup $U$ of $C_\MM(\tilde E)$ of order at least $3^{15}$. By construction,   $|\MT{11}\cap U|=1$, so  $U$ and $\{u,v\}$ together generate a $3$-subgroup of $C_\MM(\tilde E)$ of size at least  $3^{17}$.  We verify that  $g_{3b3}^{39}$ and  $h_{3b3}^{26}$ centralise $\tilde E$, and that the $\langle g_{3b3}^{39},h_{3b3}^{26}\rangle$-class of a particular element  $h_1\in U$ of order $3$ has size $72$.  By the Orbit--Stabiliser Theorem, the size of $\langle g_{3b3}^{39},h_{3b3}^{26}\rangle$ is divisible by $72$, which implies that $\tilde N$  contains a subgroup of $C_\MM(\tilde E)$ with order divisible by $8$. This centraliser is therefore at least
    $8$ times as large as the $3$-subgroup of order $3^{17}$ found above. Accounting
    for the $11232$ automorphisms of $\tilde E$ that we  enumerate by considering the conjugation action of $\tilde N$, we have $|\tilde N| \ge
        3^{17} \cdot 8 \cdot 11232 =  |\normThreeBCu|$, which completes the proof. 
    \end{proof}

    \sg{$\normThreeArk$}{norm3Ark}{Let $L_E=\{x_{3b},y_{3b},g_1,\ldots,g_4,h_4,h_4^\sigma\}$ and $L_N=\{g_{3^8},h_{3^8}\}$ be as in the accompanying code. The group $N_\MM(\langle L_E\rangle)$ is generated by $L_N$ and  isomorphic to the maximal subgroup $\normThreeArk$.
    }
    
    \begin{proof}
Lemma \ref{odd_ext} implies that $\tilde E$ has size $3^8$; we verify that the generators have order $3$ and commute, hence $\tilde E\cong 3^8$.  To prove that the elements in $L_N$ normalise $\tilde E$ is tedious, hence the accompanying code provides explicit words in elements of $L_E$ that illustrate how the elements in $L_N$ act via conjugation. This proves that $\tilde N\leq N_\MM(\tilde E)$ and also yields an explicit matrix representation of the conjugation action of $\langle L_N\rangle$ on $\tilde E$. A computation in the computer algebra system GAP \cite{gap} proves that  $\langle L_N\rangle$ acts on $\tilde E$ as the subgroup $\orth-{8}{3}.2\leq {\GL_8(3)}$.  The only maximal $3$-local subgroup in Table~\ref{max_odd_tab} whose order is divisible by $| g_{3^8} | = 41$ is  $\normThreeArk$, which tells us that $\tilde N\leq \normThreeArk$.    As for demonstrating that $L_N$ generates $\tilde N$, it remains to prove that $L_E\leq \langle L_N\rangle$, which we establish by representing the elements in $L_E$ as words in elements of $L_N$.
    \end{proof}

    \sg{$\normFiveA$}{norm5A}{
    Let $L_E=\{x_5\}$ and $L_N=\{g_5,h_5\}$ be as in the accompanying code. The group $N_\MM(\langle L_E\rangle)$ is generated by $L_N$ and  isomorphic to the maximal subgroup $\normFiveA$.
          }
    
    \begin{proof}
      The element $x_5\in\GG$ has order $5$ and  $\chi_\MM ( x_5 ) = 133$, so $x_5\in 5$A and $N_\MM(\tilde E)\cong\normFiveA$. We confirm that $L_N\leq N_\MM(\tilde E)$ and $h_5^{19}=x_5$.  The elements  $g_5^4$ and $h_5^5$ centralise $x_5$ and have orders $11$ and $19$, respectively, and it follows from the Atlas that they generate the subgroup $\HN\leq N_\MM(\tilde E)$. Lastly,  $x_5^{g_5}=x_5^2$, so $g_5$ induces an automorphism of order $4$ on $\langle x_5\rangle$. Thus, $N_\MM(\tilde E)=\langle L_N\rangle$. 
    \end{proof}

    \sgNEW{$\normFiveB$}{norm5B}{
      Let $L_E=\{x_{5b}\}$ and $L_N=\{g_{5b},h_{5b}\}$ be as in the accompanying code. The group $N_\MM(\langle L_E\rangle)$ is generated by $L_N$ and  isomorphic to the maximal subgroup $\normFiveB$.
    }{The next theorem is adapted from \cite{dlp}.}

    \begin{proof}
      We find  $c\in\MM$ such that $x_{5b}^c\in \GG$ and confirm $\chi_\MM ( x_{5b} ) = 8$, so $x_{5b} \in 5$B and $N_\MM(\tilde E)$ is isomorphic to $\normFiveB$. We also confirm $L_N\leq N_\MM(\tilde E)$. The elements  $a = h_{5b}^4$ and $g_{5b}$ centralise $x_{5b}$, so $a,g_{5b}\in \cm{x_{5b}}$. They have orders $70$ and $3$, and a product $ag_{5b}$ of order $5$. It follows that $a$ and $g_{5b}$  project to elements of orders $7$ and $3$ in $\J2$. The different orders ensure they are not inverses; in particular, their product of order  dividing $| g_{5b} a | = 5$ must project
    to an element of order $5$ exactly. Atlas information now implies that  $\langle g_{5b}, h_{5b} \rangle$ contains
    a representative for each coset of $5^{1+6}{:}2$
    in $\cma{x_5} \cong 5^{1+6}{:}2 \udot \J2$. It follows from the above that $g_{5b}^7$ is an element of order $10$
    in the normal subgroup $5^{1+6}{:}2$ of $\nma{x_{5b}}$, so it and
    its conjugates lie in in $(5^{1+6}{:}2)\setminus 5^{1+6}$. It remains to show that $5^{1+6}$ is contained in $\langle L_N\rangle$. In the accompanying code, we define  $\ell = g_{5b}^7 ( g_{5b}^7 )^{h_{5b}} \in 5^{1+6}$ and write $x_{5b}$ as a word in conjugates of $\ell$. We also define elements  $y_{5b}$, and $g_1,g_2$ and $h_1,h_2$ that  satisfy the criteria of \lem{odd_ext}, which allows us to conclude  $5^{1+6}\leq \langle L_N\rangle$. Since $h_{5b}$ induces an automorphism of order $4$ on  $\langle x_{5b}\rangle$, it follows that $\langle L_N\rangle\cong \normFiveB$.  
    \end{proof}

    \sg{$\normFiveASq$}{norm5A2}{
 Let $L_E=\{x_{5},y_5\}$ and $L_N=\{g_{5a2},h_{5a2}\}$ be as in the accompanying code. The group $N_\MM(\langle L_E\rangle)$ is generated by $L_N$ and  isomorphic to the maximal subgroup $\normFiveASq$.
    }  
    \begin{proof}
       We verify $L_N\leq N_\MM(\tilde E)$ and show that $\tilde E=5{\rm A}^2$ by checking that $x_5, y_5$ commute and that
        $x_5 \in \textup{5A}$ is conjugate to $y_5 \notin \langle x_5 \rangle$
        and to $y_5^i x_5$ for $1 \le i \le 4$ by $h_{5a2}$, $g_{5a2}^4$,
        $h_{5a2} g_{5a2}^2 h_{5a2}^2$, $g_{5a2}^2 h_{5a2} g_{5a2}$,
        and $h_{5a2} g_{5a2} h_{5a2}^2$, respectively. Thus, $L_N\subset N_\MM(\tilde E)\cong \normFiveASq$. Next, we show that
        $(( g_{5a2} h_{5a2} )^2 g_{5a2} h_{5a2}^{-1}
            g_{5a2}^{-2} h_{5a2}^{-1} )^5$ and $(h_{5a2}^2 g_{5a2} h_{5a2} g_{5a2} h_{5a2}^{-2} g_{5a2}^{-1}
            h_{5a2}^{-1} g_{5a2}^{-1})^5$ lie in   $\cma{x_5, y_5}$ which is isomorphic to $5^2 \times \unt{3}{5}$. These elements have orders $7$ and $8$, and Atlas information implies that they generate the factor $\unt{3}{5}$. The factor $5^2$ is  generated
        by $y_5 = ((g_{5a2} h_{5a2}^2)^2 g_{5a2})^{18}$ and
        $x_5 = h_{5a2}y_5 h_{5a2}^{-1}$; note that $\unt{3}{5}$ has trivial centre.
        Finally, we verify that the conjugation action of $\langle g_{5a2}, h_{5a2} \rangle$         on $\langle x_5, y_5 \rangle$ produces  $| 4 \udot 2^2 | | \sym{3} | = 16 \cdot 6 = 96$
        automorphisms of $\langle x_5, y_5 \rangle$. Thus, $N_\MM(\tilde E)=\langle L_N\rangle$.
    \end{proof}

    \sg{$\normFiveBSq$}{norm5B2}{
Let $L_E=\{x_{5b},y_{5b}\}$ and $L_N=\{g_{5b2},h_{5b2}\}$ be as in the accompanying code. The group $N_\MM(\langle L_E\rangle)$ is generated by $L_N$ and  isomorphic to the maximal subgroup $\normFiveBSq$.
    }
    
    \begin{proof}
    We have shown in the  proof of \thm{norm5B} that $x_{5b}$ and $y_{5b} \notin
        \langle x_{5b} \rangle$ are commuting elements of order $5$
    in $5^{1+6} \unlhd \nma{x_{5b}} \cong \normFiveB$. They generate
    a group $5^2 < 5^{1+6}$ containing the centre $\langle x_{5b} \rangle$.
    By \cite[\S9]{odd_local}, there are only two conjugacy classes
    of such subgroups in $\normFiveB$. One of these,  whose centralisers have
    structure $( 5 \times 5^{1+4} ){:}5^2{:}\sym{3}$, have
    as normalisers the subgroups $\normFiveBSq$ sought; the other,
    the centralisers of which have structure $5 \times 5^{1+4}{:}2^{1+4}{:}5$,
    do not. In the accompanying code we exhibit $t \in \cma{x_{5b}, y_{5b}}$ of order $3$, which proves that  $\langle x_{5b}, y_{5b} \rangle$ is of the former kind. We have confirmed $L_N\leq N_\MM(\tilde E)\cong \normFiveBSq$.  Hence we deduce $\cma{x_{5b}, y_{5b}}=( 5 \times 5^{1+4} ){:}5^2{:}\sym{3}$. All elements of order $2$ or $3$ therein project
    to elements of the same order in the factor  $\sym{3}$, and therefore generate it. We exhibit two such elements in the accompanying code.  The $5$-subgroup is constructed using  \lem{odd_ext} by exhibiting suitable elements $\ell,g_1,g_2,h_1,h_2,\sigma$. These elements are expressed as words in elements of $L_N$, which proves that  $\langle x_{5b}, y_{5b} \rangle
        \cap \cma{x_{5b}, y_{5b}}$ contains a $5$-group of order at least
    $5^{2+2+2+2} = 5^8$; this can only be the normal $5$-group  in the centraliser. 
    Finally, the conjugation action of  $\langle g_{5b2}, h_{5b2} \rangle$ on $\tilde E$ induces
    $480 = | \ling{2}{5}|$ automorphisms, which completes the proof of $N_\MM(\tilde E)=\langle L_N\rangle$.
    \end{proof}

    \sg{$\normFiveBCu$}{norm5B3}{
      Let $L_E=\{x_{5b},y_{5b},g_2\}$ and $L_N=\{g_{5b3},h_{5b3}\}$ be as in the accompanying code. The group $N_\MM(\langle L_E\rangle)$ is generated by $L_N$ and  isomorphic to the maximal subgroup $\normFiveBCu$.
    }
    
    \begin{proof}
    We verify that  $\langle x_{5b}, y_{5b}, g_2 \rangle$ is
    an elementary abelian $5$-group. By \cite[Theorem~5]{odd_local}, its normaliser  lies
    inside a maximal subgroup  $\normFiveB$, $\normFiveBSq$, $\normFiveBCu$,
    or $\normFiveBBq$. Only the third of these has order divisible by $31=|h_{5b3}^2|$, so  $\langle L_N\rangle \leq N_\MM(\tilde E)\leq \normFiveBCu$.
    We construct elements $g_1,g_2,h_1,h_1^\sigma$ and confirm that together with $x_{5b}$ and $y_{5b}$ they generate a $5$-subgroup of $\langle L_N\rangle$ of order at least $5^6$. These elements commute with $g_2$, so this $5$-subgroup lies in $C_\MM(\tilde E)$. The element $( h_{5b3} g_{5b3}^2 )^{15}$ is an involution that centralises $x_{5b}$, $y_{5b}$, and $g_2$, so it increases the size
    of the group generated by a factor of at least $2$. Lastly, we enumerate $372000=|{\rm PSL}_3(5)|$ automorphisms induced by the conjugation of $\langle g_{5b3}, h_{5b3} \rangle$ on $\tilde E$. This completes the proof that $\langle L_N\rangle=N_\MM(\tilde E)\cong \normFiveBCu$.    \end{proof}

    \sg{$\normFiveBBq$}{norm5B4}{
      Let $L_E=\{x_{5b4},y_{5b4},a,b\}$ and $L_N=\{g_{5b4},h_{5b4}\}$ be as in the accompanying code. The group $N_\MM(\langle L_E\rangle)$ is generated by $L_N$ and  isomorphic to the maximal subgroup $\normFiveBBq$.
    }
    
    \begin{proof}
    We confirm that $\tilde E$ is an abelian $5$-group of order $5^4$, so by \cite[Theorem~5]{odd_local} its normaliser is contained in a maximal subgroup of shape either
    $\normFiveB$, $\normFiveBSq$, $\normFiveBCu$, or $\normFiveBBq$. Exhibiting an element
    $[ g_{5b4}, h_{5b4} ]^6$ of order $13$
    in $\nma{L_E}$ rules out all but the last possibility, so we have confirmed that  $L_N\leq N_\MM(\tilde E)\leq \normFiveBBq$. A calculation confirms that  $x_{5b} =[ h_{5b4}^{-2}, g_{5b4}^4 ]^3$,
    $y_{5b} = x_{5b} {g_{5b4}}[ g_{5b4}^4, h_{5b4}^{-2} ]{g_{5b4}^{-1}}$,
    $a = [ h_{5b4}^{-1}, g_{5b4}^4 ]$, and
    $b = g_{5b4} h_{5b4} g_{5b4}^4 h_{5b4}^{-1} g_{5b4}^3$, so $\tilde E\leq \langle L_N\rangle$. The conjugation action of $\langle L_N\rangle$ on $\tilde E$ induces $93600=|3\times 2\cdot {\rm PSL}_2(25)|$ automorphisms, and it follows that $\langle L_N\rangle\cong \normFiveBBq$.
    \end{proof}

    \sg{$\normSevenA$}{norm7A}{
      Let $L_E=\{x_7\}$ and $L_N=\{g_{7},h_{7}\}$ be as in the accompanying code. The group $N_\MM(\langle L_E\rangle)$ is generated by $L_N$ and  isomorphic to the maximal subgroup $\normSevenA$.
    }
    
    \begin{proof}
    The argument is almost identical to that for \thm{norm5A}, using
    the facts that $x_7 \in \GG$ has order $7$; moreover $\chi_\MM ( x_7 ) = 50$,so $x_7 \in 7$A and elements
    $g_7^6, h_7^7 \in \cm{x_7} \cong 7 \times \HE$ have orders $7$ and $17$,
    while $g_7^6 \notin \langle x_7 \rangle$, $h_7^{85} = x_7$. Lastly, we note that the conjugation action of $g_7$ on $\tilde E$ induces an automorphism of order $7$, in particular
    $x_7^{g_7} = x_7^3$.
    \end{proof}

    \sgNEW{$\normSevenB$}{norm7B}{
      Let $L_E=\{x_{7b}\}$ and $L_N=\{g_{7b},h_{7b}\}$ be as in the accompanying code. The group $N_\MM(\langle L_E\rangle)$ is generated by $L_N$ and  isomorphic to the maximal subgroup $\normSevenB$.
    }{The next theorem is adapted from \cite{dlp}.}

    \begin{proof}
  We confirm that  $x_{7b} \in \GG$ has order $7$,  $\chi_\MM ( x_{7b} ) = 1$,  so $x_{7b} \in 7$B and $L_N\subset N_\MM(\tilde E)$. It follows that $N_\MM(\tilde E)\cong \normSevenB$. We first demonstrate that $\langle L_N\rangle$ contains a representative of each coset of $7^{1+4}$. For this define $a_0 = g_{7b}^6$, $a_1 = a_0^{h_{7b}}$,
    $a_3 = a_0^3 a_1^{-2} ( a_1^{h_{7b}} a_1^{h_{7b}^2} )^{-1}$,
    $y_6 = a_3^{-1} g_{7b}$, and $t = ( y_6 h_{7b} )^{14}$. We verify that $t$ has order $3$ and centralises 
    $\langle y_6, h_{7b} \rangle$. Considering conjugates in $\GG$, we verify that $|\langle y_6,h_{7b}\rangle|=10080$ and that $t \notin \langle y_6, h_{7b} \rangle$, so $\langle t, y_6, h_{7b} \rangle \le
        \langle L_N \rangle$ has $3 \cdot 10080 = | 3 \times 2.\sym{7} |$ elements,
        which is precisely the number of cosets of $7^{1+4}$ in $\normSevenB$. We show that each coset has a representative. Suppose, for a contradiction, that this is not true. In this case the intersection of
    $\langle t, y_6, h_{7b} \rangle$ with $7^{1+4}$ is a non-trivial $7$-group.
    The image of $\langle t, y_6, h_{7b} \rangle$ under
    the canonical homomorphism $\phi \colon\normSevenB \to [\normSevenB]/7^{1+4}$ then
    has order dividing $3 \cdot 10080/7 = 3 \cdot 1440$, so that the element
    $r_7 = ( y_6 h_{7b} )^6$ of order $7$
    must be mapped to the identity. On the other hand, the fact that
    $r_7 h_{7b}^{-1}$ and $h_{7b}$ have orders $20$ and $6$ implies their images under $\phi$ are
    of those orders too; but this a contradiction to $|\phi(r_7h_{7b}^{-1})|=|\phi(h_{7b}^{-1})|=|\phi(h_{7b})|=6$. Thus, it remains to prove that  $7^{1+4} \lhd \normSevenB$ is a subgroup of
    $\langle L_N  \rangle$. Note firstly that $g_{7b}^6$ lies
    in this normal subgroup, for the only alternative is that $g_{7b}$
    (of order $42$) belongs to a coset of order $42$ in the quotient group
    $[\normSevenB]/7^{1+4} \cong 3 \times 2.\sym{7}$. The absence of elements
    with this order in $2.\sym{7}$ means the $14$th power of such
    a coset would lie in the factor $3$, for which $\{ e, t, t^2 \}$ are coset representatives by the previous paragraph. But
    the orders $42$, $21$, and $21$ of $g_{7b}$, $g_{7b} t^{-1}$ and
    $g_{7b} t^{-2}$ do not divide $| 7^{1+4} |$. We therefore
    find that the elements $a_0, a_1, a_3$ defined above --- which are products of conjugates
    of $g_{7b}^6$ in $\langle L_N \rangle$ --- belong
    to $7^{1+4}$, as does $a_2 = a_1 a_1^{h_{7b}}$. Noting that
    $x_{7b} = a_0^{-1} a_1^{-1} a_0 a_1$ and $a_3$ satisfy the hypotheses of
    \lem{odd_ext} with $\ell = a_0$, suitable $g_1, h_1 \in \langle a_0,\ldots,a_3 \rangle$ and
    $\sigma \in \nma{x_{7b}, a_3}$, we deduce that they generate a subgroup of size $7^5$. It follows that $\langle L_N\rangle=N_\MM(\tilde E)$.
    \end{proof}

        \sg{$\normSevenASq$}{norm7A2}{
          Let $L_E=\{x_{7},y_7\}$ and $L_N=\{g_{7a2},h_{7a2}\}$ be as in the accompanying code. The group $N_\MM(\langle L_E\rangle)$ is generated by $L_N$ and  isomorphic to the maximal subgroup $\normSevenASq$.
    }
    
    \begin{proof}
        The proof is similar to that of  \thm{norm5A}. The element
        $x_7 \in \textup{7A}$ commutes with $y_7 \notin \langle x_7 \rangle$
        and is conjugate to $y_7$ and $x_7 y_7^i$ for $1 \le i \le 6$.
Thus $\tilde E\cong 7^2$ is 7A-pure and we confirm $L_N\subseteq N_\MM(\tilde E)\cong \normSevenASq$. The factor $7^2$ of $\cma{x_7, y_7} \cong 7^2 \times \psl{2}{7}$ can be
        handled in the same way as the $5^2$ in \thm{norm5A2} since
        $y_7 = ( g_{7a2}^3 h_{7a2}^4 g_{7a2} )^{18}$ and
        $x_7 \in y_7^{\langle g_{7a2}, h_{7a2} \rangle}$. We check  that $s = ( h_{7a2}^4 g_{7a2}^4 )^7$ and 
        $t = ( g_{7a2}^3 h_{7a2} g_{7a2} h_{7a2}^3 g_{7a2} h_{7a2} )^{14}$
        have orders $3$ and $2$, so they lie in the factor $\psl{2}{7}$. Moreover, they generate this group since  $st$ and $t$ 
       satisfy Sunday's \cite{sl_pres_1} presentation for $\psl{2}{7}$. The remaining factor of
        $| 3 \times 2.\alt{4} | | 2 | =144$ in the order of the normaliser is then accounted for by enumerating the automorphisms induced by the conjugation action of $\langle L_N\rangle$ on~$\tilde E$.
    \end{proof}

\sg{$\normSevenBSq$}{norm7B2}{
  Let $L_E=\{x_{7b},a_3\}$ and $L_N=\{g_{7b2},h_{7b2}\}$ be as in the accompanying code. The group $N_\MM(\langle L_E\rangle)$ is generated by $L_N$ and  isomorphic to the maximal subgroup $\normSevenBSq$.
    }
    
    \begin{proof}
     We have seen in the proof of
    \thm{norm7B} that  $x_{7b}$ and $a_3 \notin \langle x_{7b} \rangle$ are
    commuting elements of order $7$ such that $a_3 \in 7^{1+4} \unlhd \nma{x_{7b}}
        \cong \normSevenB$. It follows that $\langle x_{7b}, a_3 \rangle \cong 7^2$
    is a subgroup of $7^{1+4}$ containing the centre $x_{7b} \in \textrm{7B}$, so \cite[\S10]{odd_local} shows that  $\nma{L_E}$
    lies in a maximal subgroup $\normSevenB$ or $\normSevenBSq$. We show that $|g_{7b2} h_{7b2}^2|=48$, which implies that $\langle N_N\rangle\leq \normSevenBSq$. We exhibit elements $x_{7b}, a_3, g_1, h_1, h_1^{\sigma}$ such that \lem{odd_ext} implies that these elements generate a $7$-group of order at least $7^5$, centralising $\tilde E$. We demonstrate the presence of such a $7$-group
    in $\langle L_N \rangle$ by expressing each of
    these five elements as words in elements of $L_N$. Noting that $(h_{7b2} g_{7b2}^3 h_{7b2} g_{7b2}^{-1} h_{7b2}^{-3} g_{7b2}^{-1})^7$ is an element of order $3$ centralising the same $7^2$, \lem{ext_lemma} shows that the $7$-group extends
    to a subgroup of $\langle L_N \rangle \cap C_\MM(\tilde E)$
    with order at least $7^5 \cdot 3$. Finally, enumerating the conjugation action, we obtain
    \[ | \langle L_N \rangle | \ge 672 |
        \langle L_N \rangle \cap C_\MM(\tilde E)
    | \ge 672 \cdot 7^5 \cdot 3 = | \normSevenBSq|. \]
    Since $\langle L_N \rangle \le \nma{L_E} \le \normSevenBSq$, equality follows.
    \end{proof}

    \sgNEW{$\normSevenBSqOut$}{norm7B2'}{
      Let $L_E=\{x_{7b},y_{7b}\}$ and $L_N=\{x_{7b},y_{7b},x_4,x_{14}\}$ be as in the accompanying code. The group $N_\MM(\langle L_E\rangle)$ is generated by $L_N$ and  isomorphic to the maximal subgroup $\normSevenBSqOut$.
    }{The next theorem is adapted from \cite{four_fus}.}

    \begin{proof}
    Arguing as in the proof of  \thm{norm7A2} establishes that
    $\tilde E\cong 7^2$ is 7B-pure; conjugating elements are provided in the code. This allows us to confirm that  $L_N\leq N_\MM(\tilde E)\leq \normSevenBSqOut$. We verify that $x_4$ and $x_{14}$ satisfy a presentation for the group  $\lins{2}{7}$, which implies the claim. 
    \end{proof}

        \sgNEW{$\normElevenSq$}{norm11_2}{
          Let $L_E=\{x_{11},y_{11}\}$ and $L_N=\{x_{11}, y_{11}, x_3, x_4, x_5 \}$ be as in the accompanying code. The group $N_\MM(\langle L_E\rangle)$ is generated by $L_N$ and  isomorphic to the maximal subgroup $\normElevenSq$.
    }{The next theorem is adapted from  \cite{four_fus}.}
    
    \begin{proof}
        We confirm that $\tilde E\cong 11^2$ and  $L_N\subset N_\MM(\tilde E)\cong \normElevenSq$, see \prp{bigoddlocalprop}; note that all elements of order $11$ in $\MM$ are conjugate. We now check that $x_3$, $x_4$ and $x_5$
        have the orders indicated by their subscripts, and that $x_5$ commutes
        with $x_3, x_4$. Furthermore, the cosets of $x_4$ and $x_3$
        modulo $\langle x_4^2 \rangle$ satisfy the presentation
        $\langle a, b \mid a^2, b^3, ( ab )^5 \rangle$
        for the simple group $\alt{5}$, which by Von Dyck's Theorem~\cite[Theorem 2.53]{hcgt} ensures
        $\langle x_3, x_4 \rangle / \langle x_4^2 \rangle \cong \alt{5}$.
        Observing that $x_4^2$ commutes with $x_3$ and (of course) $x_4$,
        it must be that $\langle x_3, x_4, x_5 \rangle\cong 5 \times 2.\alt{5}$, so $\langle L_N \rangle =   N_\MM(\tilde E)$.
    \end{proof}

    \sgNEW{$\normThirteenA$}{norm13A}{
Let $L_E=\{g_{13}\}$ and $L_N=\{g_{13}, y_{12}, c, d \}$ be as in the accompanying code. The group $N_\MM(\langle L_E\rangle)$ is generated by $L_N$ and  isomorphic to the maximal subgroup $\normThirteenA$.
    }{The next theorem is adapted from \cite{dlp}.}

    \begin{proof}
    We verify that $L_N\subset N_\MM(\tilde E)$. Enumerating a conjugate of $\langle c,d\rangle$ in $\GG$ proves that $|\langle c,d\rangle|=5616$, so $g_{13}$ cannot be a 13B element: the size of the normaliser of the latter is not divisible by $5616$. Thus, $g_{13}$ is a 13A element, hence $N_\MM(\tilde E)\cong \normThirteenA$. We confirm that $\cm{g_{13}} \cong 13 \times \psl{3}{3}$ is contained
    in $\langle g_{13}, y_{12}, c, d \rangle$ by verifying that  $c, d$ centralise the non-commuting elements $g_{13}$ and $y_{12}$.
    Thus, $g_{13}$ is an element of order $13$ not  belonging
    to $\langle c, d \rangle < \cm{y_{12}}$, and  \lem{ext_lemma} implies
    that $| \langle g_{13}, c, d \rangle |
        \ge 13 | \langle c, d \rangle| = 13 \cdot 5616 =
        | 13 \times \psl{3}{3} |$. The observation that $| \langle g_{13}, c, d \rangle |\leq \cm{g_{13}}$ yields the desired result. The conjugation action of $\langle y_{12}\rangle$ on $\tilde E$ gives rise to $12$ automorphisms; now $\langle L_N\rangle=N_\MM(\tilde E)$ follows.
    \end{proof}

\sg{$\normThirteenB$}{norm13B}{
  Let $L_E=\{c\}$ and $L_N=\{g_{13}, c, c_2, x_1, x_2  \}$ be as in the accompanying code. The group $N_\MM(\langle L_E\rangle)$ is generated by $L_N$ and  isomorphic to the maximal subgroup $\normThirteenB$.
    }
    
\begin{proof}
  The element $c$ has order $13$ and a conjugate in $\GG$. The Atlas implies that $\GG$ has no 13A elements, so $c$ lies in class 13B and we confirmed that $L_N\subset N_\MM(\tilde E)\cong \normThirteenB$. Enumerating a conjugate of  $\langle x_1, x_2 \rangle$ in $\GG$ shows that $|\langle x_1,x_2\rangle|= 288 =| 3 \times 4.\alt{4}|$, so this group must be  $3 \times 4.\alt{4}$. The elements $g_{13}, c, c_2$
    in turn have orders coprime to $288$ and thus lie in the $13$-group
    $13^{1+2}$. An enumeration proves that $g_{13}, c, c_2$
    generate $13^{1+2}$. 
    \end{proof}

\sg{$\normThirteenBSq$}{norm13B2}{
 Let $L_E=\{c,c_2\}$ and $L_N=\{c, c_2, x_1, x_2, x_4 \}$ be as in the accompanying code. The group $N_\MM(\langle L_E\rangle)$ is generated by $L_N$ and  isomorphic to the maximal subgroup $\normThirteenBSq$.
    }
    
    \begin{proof}
    The proof of \thm{norm13B} shows that $\langle c, c_2 \rangle \cong 13^2$; moreover, as
    $c \in \textup{13B}$ is conjugate to $c_2$ and $cc_2^i$ for 
    $1 \le i \le 13$ with conjugating elements $x_2^3$ and $x_2^3 x_1^{12i + 13} x_2^3$,
    respectively, this group is 13B-pure. We confirm that $L_N\subset N_\MM(\tilde E)\cong \normThirteenBSq$. We verify that the conjugation action of $\langle x_1, x_2, x_4 \rangle$ on $\tilde E$ induces $4 | \lins{2}{13} | =8736$ automorphisms; this implies that $\langle L_N\rangle=N_\MM(\tilde E)$.
    \end{proof}

    \sg{$41{:}40$}{norm41}{
      Let $L_E=\{g_{41}\}$ and $L_N=\{g_{41},h_{41}\}$ be as in the accompanying code. The group $N_\MM(\langle L_E\rangle)$ is generated by $L_N$ and  isomorphic to the maximal subgroup $41{:}40$.
    }
    
    \begin{proof}
   We show that  $g_{41}$ has order 41 and
    $g_{41}^{h_{41}} = g_{41}^{22}$ and $22$ is a primitive root modulo $41$.
    \end{proof}

   \section{\label{2local}The Maximal 2-Local Subgroups of $\MM$}
\noindent The maximal $2$-local subgroups of $\MM$,
all known to be maximal before the publication of the Atlas \cite{atlas}, were first classified by Meierfrankenfeld and Shpectorov, see \cite{2local2,2local1}. We introduce the necessary
terminology before re-stating their results,  which we need to do in order to justify our constructions. Recall that $z$ denotes the central 2B involution in $\GG=C_\MM(z)$, and that $\GG\cong \QQ\udot \CO$ where $\QQ=\QStr$. If $g_2\in\MM$ is a 2B involution and $u,v\in\MM$ satisfy $g_2^u=z=g_2^{v}$, then $v^{-1}u\in C_\MM(z)$ and therefore $\QQ^{v^{-1}u}=\QQ$. Thus, the group  $\QQ^{v^{-1}}=\QQ^{u^{-1}}$ is independent of the choice of $u$ or $v$, and we can define $\QQ_{g_2}$ as $\QQ^{u^{-1}}$ More generally, for a subgroup $U\leq \MM$ we write \[\QQ_U= \bigcap_{x \in U \cap \textrm{2B}} \QQ_x.\]
   A 2B involution $g_2$ is  \textsl{perpendicular} to $g_1$ if $g_1 \in \QQ_{g_2}$.
    A \textsl{singular} subgroup of $\MM$ is an elementary abelian 2-group
    in which all involutions are of class 2B and pairwise perpendicular. There are two conjugacy classes of singular subgroups
    of order $2^5$ in $\MM$, denoted as ``type 1'' and ``type 2''
    in accordance with \cite{2local1}, see also Section \ref{sing_sub}. An \textsl{ark} $A$ is a group
    generated by a type-1 2B$^5$ singular subgroup $U$ and
    all the type-2 2B$^5$ subgroups that intersect $U$
    in an index 2 subgroup; it turns out that the size of an ark is always $| A | = 2^{10}$.
    With these definitions, Meierfrankenfeld and Shpectorov prove the following result.
        
    \begin{proposition}[\!\!{\cite[Theorem~1]{2local1}} and {\cite[Theorem~A]{2local2}}]
        \label{2localprop}
        The Monster contains exactly 7 conjugacy classes
        of maximal 2-local subgroups. They are:
        \begin{ithm}
            \item[\rm (1)] the normalisers of subgroups of types {\rm 2A} and
                {\rm 2A$^2$}, with structures $\normTwoA$ (where $\BB$
                is the Baby Monster) and $\normTwoASq$ respectively;
            \item[\rm (2)] the normalisers of singular subgroups of types
                {\rm 2B}, {\rm 2B$^2$}, {\rm 2B$^3$}, and {\rm 2B$^5$} type-2, which
                have structures $\normTwoB$,
                $\normTwoBSq$, $\normTwoBCu$, and $\normTwoBQu$
                respectively; and
            \item[\rm (3)] the normalisers of arks, with structure $\normTwoArk$.
        \end{ithm}
    \end{proposition}

\begin{remark}
Generators for the centralisers of {\rm 2A} and {\rm 2B} involutions described in Proposition \ref{2localprop} already appear with proof in \cite[\S 4]{dlp} and \cite[\S 2.5]{dlp} respectively.  For completeness, these generators are included in the accompanying code \cite{pisgit} along with the relevant arguments, but we do not reproduce them here.
\end{remark}

    \subsection{Preliminary results}
    \subsubsection{\label{sing_sub}Checking singularity of 2B-pure subgroups}
    
    The construction of several subgroups in \prp{2localprop}
    demands some means of identifying singular subgroups and classifying
    them up to conjugacy; otherwise, we have no means of ascertaining that what
    we produce are in fact maximal subgroups. Fortunately, Meierfrankenfeld
    and Shpectorov provide some simple tests for this.
  
    \begin{proposition}
      \label{sing_class}  \label{sing_perp}The following hold.
     \begin{ithm} 
\item        There are exactly $6$ conjugacy classes of
        non-trivial singular subgroups in $\MM$; the corresponding orders
        are $2$, $2^2$, $2^3$, $2^4$, $2^5$, and $2^5$.
      \item    The perpendicularity relation is symmetric.
    \item   A subgroup $U$ is singular if and only if it is generated
        by a set $\mathcal{U}$ of pairwise perpendicular {\rm 2B} involutions.
        Furthermore, the group $\QQ_U = \bigcap_{x \in U \cap \textrm{\rm 2B}} \QQ_x$ coincides with $\bigcap_{x \in \mathcal{U}} \QQ_x$.
     \end{ithm}
    \end{proposition}
    \begin{proof}
      Part a) is \cite[Proposition~4.15]{2local1}, part b) is \cite[Lemma~4.1]{2local1}, and part c) is \cite[Lemma~4.3]{2local1}.
    \end{proof}
    
    Part a) of Proposition \ref{sing_class} shows that, apart from the $2^5$ case (see  \S\ref{2_5_arks}),
    the conjugacy class of a singular subgroup is determined by the group size. Singularity can be tested
    by straightforward enumeration of the group elements and ascertaining
    the non-identity elements are commuting perpendicular 2B involutions; this is a feasible task since \mmgroup\ provides functionality to conjugate a 2B involution to $z\in \GG$ and to test whether an element lies in $\QQ$. Part c) provides some simplifications for this test. This approach  suffices for the construction of
    the maximal 2A$^2$, 2B$^2$, and 2B$^3$ normalisers.

    \subsubsection{\label{2_5_arks}The 2B$^5$ and Ark Normalisers}
    
    To construct the 2B$^5$ and ark normalisers, the question
    of distinguishing the two conjugacy classes of singular subgroup
    with structure $2^5$ must be addressed. The following result forms
    the basis of the method that we employ. 
    
    \begin{proposition}[\!\!{\cite[Lemma~4.14]{2local1}}]
        \label{the2_5s}
        The following hold.
        \begin{ithm}
            \item If $U\leq \MM$ is singular $2^5$ type-1, then
                $\QQ_U /U$ is of order 2.  Furthermore,
                all involutions in $\QQ_U \setminus U$ are
                in class {\rm 2A}.
            \item If $U\leq \MM$ is singular $2^5$ type-2, then $\QQ_U = U$.
        \end{ithm}
    \end{proposition} 
    
    The task of determining $\QQ_U$ can be simplified by the observation
    that, if $U$ is generated by a subgroup $V< U$ and an involution $t$ perpendicular
    to the generators of $V$, then \prp{sing_perp} may be applied
    to yield $\QQ_U = \QQ_V \cap \QQ_t$.


    \sg{$\normTwoASq$}{norm2A2}{ Let $L_E=\{y,y^{g_2} \}$ and $L_N=\{g_2,h_2\}$ be as in the accompanying code. The group $N_\MM(\langle L_E\rangle)$ is generated by $L_N$ and isomorphic to the maximal subgroup $\normTwoASq$.

    }
    
    \begin{proof}
We confirm that $\tilde E$ is $2\mt{A}$-pure of size $2^2$, and so $\tilde N\cong \normTwoASq$. Moreover, we confirm that $L_N \leq\tilde N =   N_\MM(\langle L_E\rangle)$.
Define
    $a = h_2^2 g_2^9$ and $b = h_2^2 g_2^{12}$.  We compute that $\langle a, b \rangle \le C_\MM (
      \tilde E) \cong 2^2\udot {}^2 \lie{E}{6}{2}$. Taking the image under the quotient by $\tilde E$, we have $\langle \tilde E a, \tilde E b \rangle \le
        {}^2 \lie{E}{6}{2}$. Now $a^2, b^2 \notin \tilde E$ satisfy $( a^2 )^{13} = 1$
    and $( b^2 )^{19} = 1$, so that $\tilde Ea^2$ and
    $\tilde E b^2$ must have orders 13 and 19, respectively.
    Since the sole maximal subgroup of ${}^2 E_6 ( 2 )$
    with order divisible by 19 has the form $U_3 ( 8 ) : 3$
    (see \cite{2e62}), and the order of this group
    is not divisible by 13, the elements $\tilde E a, \tilde E b$ must
    in fact generate $C_{\MM} (\tilde E ) / \tilde E$. 
    Now $\tilde E \leq \langle L_N \rangle$ since $y^{g_2} = ( h_2 g_2 )^{14}$,  and hence $C_{\MM} (\tilde E) \leq \langle L_N \rangle$.  Finally,  we note that the conjugation action of $\langle L_N\rangle$ induces 6 = $|S_3|$ distinct automorphisms of $\tilde E$, implying that $N_{\MM}(\tilde E) = \langle L_N\rangle$.
    \end{proof}
    
    
    \sg{$\normTwoBSq$}{norm2B2}{
Let $L_E=\{z,z_1 \}$ and $L_N=\{g_{2b}, h_{2b}\}$ be as in the accompanying code. The group $N_\MM(\langle L_E\rangle)$ is generated by $L_N$ and isomorphic to $\normTwoBSq$, a maximal subgroup of $\MM$.
    }
    
    \begin{proof}
    We verify that $\tilde E$ is $2\mt{B}$-pure of size $2^2$, and so $\tilde N\cong \normTwoBSq$. 
    Moreover, $\tilde E \leq \QQ \leq \GG$ and $L_N \leq\tilde N =   N_\MM(\langle L_E\rangle)$.
   We now deviate from Procedure \ref{provesg}. To show that $\langle g_{2b}, h_{2b} \rangle$ is the complete normaliser of $\tilde E$, it suffices to prove that $| \langle g_{2b}, h_{2b} \rangle |
    \ge | N_\MM ( \tilde E) |$.  
    Define four subsets $\mathcal{F}_1=L_N, \mathcal{F}_2, \mathcal{F}_3,
    \mathcal{F}_4$ of $\langle L_N \rangle$, and let  $K_i = \langle \mathcal{F}_i, \dots, \mathcal{F}_4 \rangle$ for $i\in\{1,\ldots,4\}$.
  We will describe homomorphisms $\varphi_1, \varphi_2, \varphi_3$ such that $K_{i+1} \le \ker{\varphi_i}$ and
    $\ker{\varphi_3} < \ker{\varphi_2} < \ker{\varphi_1}$
     for $i = 1, 2, 3$. We can then find a lower bound for $|\langle g_{2b}, h_{2b} \rangle|$ using a repeated application of the Isomorphism Theorem.
    
  First, let $\varphi_1\colon  K_1 \to \mathrm{Aut}(\tilde E)$ be the homomorphism induced by conjugation.
    By checking the generators of each  $K_i$, we verify that $|\varphi_1(K_1)|=6=|\sym{3}|$,  while $K_2 \leq C_{\MM}(\tilde E)\leq \GG = C_\MM(z)$; this also shows that $K_2 \leq  \ker{\varphi_1}$.

   Now consider the action of $K_2$ on  $\QQ_{\tilde E}/{\tilde E}$, as defined in Lemma \ref{sing_class}. By \cite[Lemma~4.5]{2local1}, the image of this action is  isomorphic
    to some subgroup of $\MT{24}$. Let
    $\varphi_2\colon K_2 \to \mathrm{Aut}({\QQ_{\tilde E}/{\tilde E}})$ be the corresponding action homomorphism. In order to compute this
    homomorphism, we give (and verify) generators for $\QQ_{\tilde E}$ in the accompanying code.
    It is then easily shown that $K_3 \le
    \ker{\varphi_2}$.  Moreover,
    we claim that $\varphi_2 ( K_2 ) \cong \MT{24}$. We construct elements in $\varphi_2 ( K_2 )$ with orders 5, 23, and 21. No proper subgroup of $\MT{24}$ has elements of all three of these orders, so the claim holds.

Now observe that $\mathcal{F}_3, \mathcal{F}_4 \subseteq
        \GG \cong \QQ \udot \CO$ and consider
    $\varphi_3\colon K_3 \to \CO$, the restricted canonical homomorphism from $\GG$ into $\CO$.  We find that $| \varphi_3 (
        \langle \mathcal{F}_3 \rangle
    ) | = 2^{11}$ while $K_4 \le \ker{\varphi_3}$. Finally, with the observation that $\mathcal{F}_4 \subseteq \QQ$,
    a direct calculation shows that 
    $| K_4 | = 2^{24}$. Lastly,   $\langle L_N \rangle  = | N_\MM ( {\tilde E} ) |$ follows from a repeated application of the Isomorphism Theorem, which reveals that
 \[
 \langle L_N \rangle \geq |K_4| \prod\nolimits_{i=1}^3 | \im{\varphi_i}| \geq  2^{24}.6.|\MT{24}|. 2^{11}    = | \normTwoBSq | \\
         = | N_\MM ( {\tilde E} )|.\qedhere
 \]
    \end{proof}
    
    \sg{$\normTwoBCu$}{norm2B3}{
   Let $L_E=\{z,z_1,w \}$ and $L_N=\{g_{3}, h_{3}\}$ be as in the accompanying code. The group $N_\MM(\langle L_E\rangle)$ is generated by $L_N$ and isomorphic to the maximal subgroup $ \normTwoBCu$.
    }
    
    \begin{proof}
    We verify that $\tilde E$ is $2\mt{B}$-pure of size $2^3$, and so $ N_\MM(\langle L_E\rangle)$ is of the required form. Moreover,  $L_N \leq   N_\MM(\langle L_E\rangle)$. We proceed as in Theorem~\ref{thm_norm2B2} to show  $\langle L_N\rangle= N_\MM(\langle L_E\rangle)$.  Consider
    four subsets $\mathcal{F}_1 = \{ g_3, h_3 \}$,
    $\mathcal{F}_2$, $\mathcal{F}_3$, and $\mathcal{F}_4$ of
    $\langle g_3, h_3 \rangle$, and, for each $1\leq i \leq 4$, let $K_i = \langle \mathcal{F}_i,  \dots,  \mathcal{F}_4 \rangle$. Let $\varphi_1\colon K_1 \to \mathrm{Aut}({{\tilde E}})$ be induced by conjugation; we compute that $| \varphi_1 ( \langle \mathcal{F}_1 \rangle ) |
    = 168$ while $K_2 \le \ker{\varphi_1}$. In particular, $K_2 \le
    C_\MM ( {\tilde E} )$. Now consider the action of $K_2$ on $\QQ_{\tilde E}/{\tilde E}$.  Since $K_2$ centralises $\tilde E$, it follows from \cite[Lemma~4.8]{2local1} that elements of
    $K_2$ permute $\QQ_{\tilde E}/{\tilde E}$ by conjugation, inducing a group isomorphic
    to some subgroup of $3 .\sym{6}$. Let 
    $\varphi_2\colon K_2 \to \mathrm{Aut}({\QQ_{\tilde E}/{\tilde E}})$ be the corresponding action homomorphism.
To compute images under this
    homomorphism,  verified generators for $\QQ_{\tilde E}$ are given in the accompanying code.
 We find that
    $| \varphi_2 ( \langle \mathcal{F}_2 \rangle ) |
    = 2160 = |3 .\sym{6}|$ while $K_3 \le \ker{\varphi_2}$. Noting that $\mathcal{F}_3, \mathcal{F}_4 \subseteq
        \GG \cong \QQ \udot \CO$, define
    $\varphi_3\colon K_3 \to \CO$ to be the restricted canonical homomorphism from $\GG$ into $\CO$. We find that $K_4 \le \ker{\varphi_3}$ and  $| \varphi_3 ( \langle \mathcal{F}_3 \rangle )
    | \ge  2^{16}$.
  Moreover by a direct calculation,
    $| \langle \mathcal{F}_4 \rangle | = 2^{23}$.
The claim now follows from a repeated application of the Isomorphism Theorem, which shows that
 \[
 \langle L_N \rangle \geq |K_4| \prod\nolimits_{i=1}^3 | \im{\varphi_i}| \geq  2^{23}.168.|3 .\sym{6}|. 2^{16}   = | \normTwoBCu | \\
         =| N_\MM ( {\tilde E} )|.\qedhere
 \]
    \end{proof}

  \subsection{The group $\normTwoBQu$.}

\begin{theorem}\label{thm_norm2B5}
 Let $L_E=\{z,z_1,w,k_4, k_5 \}$ and $L_N=\{g_{5}, h_{5}\}$ be as in the accompanying code. The group $N_\MM(\langle L_E\rangle)$ is generated by $L_N$ and isomorphic to the maximal subgroup $\normTwoBQu$.
\end{theorem}

    \begin{proof}
We confirm that $\tilde E$ is $2\mt{B}$-pure of size $2^5$,
and that $| \QQ_{\tilde E} | = |{\tilde E}| = 2^5$. Hence,  $\tilde E$ is a singular $2\mt{B}^5$ subgroup of type 2 by \prp{the2_5s},  so $ N_\MM(\langle L_E\rangle)$ is maximal of the required shape by  \prp{2localprop}.
We also find that $L_N \leq   N_\MM(\langle L_E\rangle)$. To prove that $\langle L_N \rangle \leq   N_\MM(\tilde E)$, we proceed as in Theorem~\ref{thm_norm2B2}. 
Define subsets $\mathcal{F}_1 = \{ g_5, h_5 \}$,
    $\mathcal{F}_2$, $\mathcal{F}_3$, and $\mathcal{F}_4$ of
    $\langle L_N \rangle$, and, for $1\leq i \leq 4$, write $K_i = \langle \mathcal{F}_i,  \dots,  \mathcal{F}_4 \rangle$. Let $\varphi_1\colon K_1 \to \mathrm{Aut}({{\tilde E}})$ be the action homomorphism for $K_1$ acting by conjugation on $\tilde E$.  We verify that  $K_2 \le \ker{\varphi_1}$, so $K_2 \le
    C_\MM ( {\tilde E} )$.  We further claim that $ \varphi_1 ( \langle \mathcal{F}_1 \rangle ) 
    \cong \ling{5}{2} = \psl{5}{2}$.  We find a pair of elements with images under $\varphi_1$ of orders 15 and 31 respectively. Since $\psl{5}{2}$ has no proper subgroup divisible by both of these orders, the claim follows. To introduce the next homomorphism $\varphi_2$,  it is necessary to define the subgroup $V = \langle z,z_1,w,k_4 \rangle < \tilde E$ of order $2^4$.  We verify that $V$ is singular, and that $|\QQ_V| = 2^7$. Now $K_2 \leq  C_\MM ( {\tilde E} )\leq C_\MM (V )$, and so \cite[Lemma~4.10]{2local1} implies that $K_2$ permutes $\QQ_V/V$, inducing a group isomorphic to some subgroup of $\sym{3}$. Let 
    $\varphi_2\colon K_2 \to \mathrm{Aut}({\QQ_{V}/{V}})$ be the corresponding action homomorphism.  We compute that $| \varphi_2 ( \langle \mathcal{F}_2 \rangle )
    | = 6$ while $K_3 \le \ker{\varphi_2}$.

Now $\mathcal{F}_3, \mathcal{F}_4 \subseteq
        \GG \cong \QQ \udot \CO$,  so consider the restriction 
    $\varphi_3\colon K_3 \to \CO$ of the canonical homomorphism from $\GG$ into $\CO$.  We calculate that $K_4 \le \ker{\varphi_3}$ and that $| \varphi_3 ( \langle \mathcal{F}_3 \rangle )
    | \ge  2^{14}$.  Moreover, 
  a direct calculation shows that 
    $| \langle \mathcal{F}_4 \rangle | = 2^{21}$. Applying the Isomorphism Theorem, we deduce that the following, which proves the claim.
  \[
 \langle L_N \rangle \geq |K_4| \prod\nolimits_{i=1}^3 | \im{\varphi_i}| \geq  2^{21}.|\psl{5}{2}|.6. 2^{14}     = | \normTwoBQu | \\
         = | N_\MM ( \tilde E ) |.\qedhere
 \]
    \end{proof}

    
    \sg{$\normTwoArk$}{normArk}{
 Let $L_E=\{z,z_1,w,k_4, k_5,k_4^{\rho^{16}},
            k_5^{\rho}, k_5^{\rho^{15}}, k_5^{\rho^{16}},
            k_5^{\rho^{30}} \}$ and $L_N=\{g_{10}, h_{10}\}$ be as in the accompanying code. The group $N_\MM(\langle L_E\rangle)$ is generated by $L_N$ and isomorphic to the maximal subgroup $\normTwoArk$.    
    }
    
    \begin{proof}
We first verify that $\tilde E = \langle L_E \rangle$ is an \textit{ark}, as in the definition preceding Proposition \ref{2localprop}.  We first verify that $T_1 = \langle z, z^{\tau}, w, k_4, k_4^{\rho^{16}} \rangle$ is a singular subgroup of order $2^5$ and moreover that $| \QQ_{T_1} | = 64 = 2 \cdot 2^5$.  It follows
    from \prp{the2_5s} that $T_1$ is a singular $2^5$ subgroup
    of type 1.
    
    To extend this to ${\tilde E} = \langle L_E \rangle$, recall from the proof
    of Theorem~\ref{thm_norm2B5} that $\langle z, z^{\tau}, w, k_4, k_5 \rangle < {\tilde E}$
    is a singular 2B$^5$ of type 2. It contains the index 2 subgroup $V = \langle z,z_1,w,k_4 \rangle$ of ${\tilde E}$. Now 
 \cite[Corollary~4.12]{2local1} asserts that each singular 2B$^4$
    belongs to a unique 2B$^5$ of type 2 (and two of type 1).  
    On the other hand, there exists an element $\rho$ of order 31 (defined in the accompanying code ) which normalises both ${\tilde E}$ and $T_1$, but not $V < T_1$.  We compute that $\rho$ acts transitively on singular 2B$^5$s of type 2 that meet $T_1$ in a singular 2B$^4$. Hence $\tilde{E}$ is an ark, so $N_\MM ( {\tilde E} )$ is maximal by  \prp{2localprop}. 
      
We next confirm that $g_{10}, h_{10} \in N_\MM ( {\tilde E} )$.  To show that $ \langle L_N \rangle = N_\MM ( {\tilde E} )$,  we proceed in a similar way to Theorem~\ref{thm_norm2B2}.
Consider three subsets
    $\mathcal{F}_1 = \{ g_{10}, h_{10} \}$, $\mathcal{F}_2$,
    and $\mathcal{F}_3$ of $\langle g_{10}, h_{10} \rangle$, and for $1\leq i \leq 3$ define $K_i = \langle \mathcal{F}_i,  \dots ,
    \mathcal{F}_3 \rangle$. Let $\varphi_1\colon K_1 \to \mathrm{Aut}({{\tilde E}})$ be induced by conjugation; we find that $K_2 \le \ker{\varphi_1}$.  Moreover,
    we claim that $\varphi_1 ( K_1 ) \cong \orth+{10}{2}$.
    It follows from the structure of $N_\MM ( \tilde E )$
    (cf.\  \cite[Lemma~5.18]{2local1}) that
    $\varphi_1 ( K_1 ) \le \orth+{10}{2}$.  We find permutations in $\varphi_1 ( K_1 )$ of orders 17 and 31. Since 
 $\orth+{10}{2}$ has no maximal subgroups with order divisible
    by both 17 and 31, the claim follows.

    At this point, we note that the elements $x_{60}, r$ and $s$ defined in the accompanying code
    belong to $N_{\MM} ( \tilde E )$ and satisfy
    $g_{10} = ( x_{60}^{30} )^r$ and
    $h_{10} = ( x_{60}^3 )^s$. Applying $\varphi_1$ reveals that
    $x_{60}$ and $g_{10} h_{10}$ project to elements of order 60 and 21,
    respectively, in $\orth+{10}{2}$. Since the GAP Character Table Library \cite{gap_ctbllib}
    shows that all elements of order 60 in this group power up
    to conjugacy classes 2A and 20A thereof, $g_{10}$ and $h_{10}$ are
    standard generators in the sense of \cite{online_atlas}.
    
Note that $\mathcal{F}_2, \mathcal{F}_3 \subseteq
        \GG \cong \QQ \udot \CO$,  and consider
    $\varphi_2\colon  K_2 \to \CO$ defined as the restriction of the canonical homomorphism from $\GG$ into $\CO$.
    It is easily verified that
    $| \varphi_2 (
        \langle \mathcal{F}_2 \rangle
    ) | = 2^9$ while $K_3 \le \ker{\varphi_2}$.
    Finally,  a direct calculation shows that 
    $| \langle \mathcal{F}_3 \rangle | = 2^{17}$. Applying the Isomorphism Theorem, we deduce the following, which completes the proof.
  \[
 \langle L_N \rangle \geq |K_3| \prod\nolimits_{i=1}^2 | \im{\varphi_i}| \geq  2^{17}.|\orth+{10}{2}|.2^9     = | \normTwoArk | \\
         = | N_\MM( {\tilde E} ) |.\qedhere
 \]
    \end{proof}

    
    \section{\label{a5s}$\alt{5}$ in the Monster}
    
    \noindent The construction of the non-local subgroups of $\MM$ requires
    some preliminary results on subgroups of $\mathbb{M}$ isomorphic to $\alt{5}$.
    With the exception of $\pgl{2}{13}$, all non-local maximal subgroups contain
    a subgroup of this shape, and considering how each $\alt{5}$ can be extended
    to larger subgroups of $\MM$ has conversely played a significant role
    in the existing classifications of some of the maximal subgroups.
    
    The conjugacy classes of subgroups isomorphic to $\alt{5}$ in $\MM$ were classified by Norton \cite{anatomy}. There are eight conjugacy classes, as listed in Table~\ref{a5_classes}. Six classes are uniquely identified by the (unique) $\MM$-classes to which
    their elements of orders $2$, $3$, and $5$ belong, while the two containing
    2B, 3B, and 5B elements may be distinguished using their normalisers.
    We follow Norton in labelling the former sextet by their class fusions
    and the latter pair, which respectively occur as subgroups of
    the double cover of the Baby Monster and Thompson group, ``B'' and~``T''.  Instances of several of these types of $\alt{5}$ are constructed
    in \cite{dlp}. The following lemma reproduces these results,
   with the addition of a subgroup $\alt{5}$ of type ACA and a simplification
    due to \cite{four_fus} in the proof for the type B case.
    
    \begin{table}[t] 
            \begin{tabular}{llll}
                \toprule
                Type & Class Fusions & Centraliser & Normaliser \\
                \midrule
                AAA & 2A, 3A, 5A & $\alt{12}$ & $\normAAA$ \\
                BAA & 2B, 3A, 5A & $2.\MT{22}.2$ & \\
                BBA & 2B, 3B, 5A & $\MT{11}$ & $\MT{11} \times \sym{5}$ \\ 
                ACA & 2A, 3C, 5A & $\unt{3}{8}{:}3$ & $\normACA$ \\
                BCA & 2B, 3C, 5A & $2^{1+4} (\alt{4} \times \alt{5})$ & \\
                BCB & 2B, 3C, 5B & $\dih{10}$ & \\
                B & 2B, 3B, 5B & $\sym{3}$ & $\alt{5}{:}4$ \\
                T & 2B, 3B, 5B & 2 & $\sym{5} \times \sym{3}$ \\
                \bottomrule
            \end{tabular}\label{a5_classes}%

            \vspace*{1ex}
            
            \caption{\small{The conjugacy classes of $\alt{5}$ in $\MM$, the $\MM$-classes
            to which their unique conjugacy classes of elements of orders 2, 3,
            and 5 fuse, their centralisers, and (where stated) their normalisers; this table is adapted from Norton \cite{anatomy}, as well as \cite[p. 234]{atlas}.}}
    \end{table}

    \begin{lemma}
      \label{a5_lemma}
        For each type $\text{\rm UVW}\in\{\text{\rm AAA, BAA, BBA, ACA, BCA, BCB, B, T}\}$, the elements $g_{2UVW}$ and $g_{3UVW}$ as given in the accompanying code generate a subgroup $\alt{5}$ of $\MM$ of type ${\rm UVW}$.
    \end{lemma}
    \begin{proof}
      In each case, we verify that the given generators satisfy a presentation for $\alt{5}$. The generators are non-trivial, so by  Von Dyck's Theorem and the simplicity of $\alt{5}$, they generate a subgroup $\alt{5}$. Since the conjugacy classes in $\MM$ of
        elements of orders $2$, $3$, and $5$ are distinguished by their values
        under $\chi_\MM$, we can confirm the type of the $\alt{5}$, with the exception of types B and T.  To complete the proof, it must be shown that
        $\langle g_{2B}, g_{3B} \rangle$ is not of type T and
        $\langle g_{2T}, g_{3T} \rangle$ not of type B. The latter case
        is easier: in the code we exhibit an element $y_3$ that has order $3$ and centralises
        $\langle g_{2T}, g_{3T} \rangle$, whereas all elements of order $3$
        in the normaliser $\alt{5}{:}4$ of an $\alt{5}$ of type B  lie
        in the $\alt{5}$ and therefore do not  centralise it.
         As for $\langle g_{2B}, g_{3B} \rangle$, the normaliser of
        an $\alt{5}$ of type T has structure $\sym{5} \times \sym{3}$ and cannot
        contain an element of order $4$ whose square lies  outside the
        $\alt{5}$. In the code we exhibit an element $y_4$ with these properties: we verify that it has   order $4$, commutes with $g_{2B}$, and satisfies
        $g_{3B}^{y_4} = g_{2B} g_{3B}^2 g_{2B} g_{3B} g_{2B} g_{3B}^2$; in particular, $y_4$ belongs to the normaliser of $\alt{5}$. On the other hand, ${y_4^2}$ centralises   $\langle g_{2B}, g_{3B} \rangle \cong \alt{5}$; since $\alt{5}$
        has a trivial centre,  $y_4^2 \notin
            \langle g_{2B}, g_{3B} \rangle$, as required.
    \end{proof}
    
   \begin{remark}\label{good_a5}Our subgroup of type T in  \lem{a5_lemma} is different to that given
        in \cite{dlp}; ours is chosen to intersect a conjugate of our $\alt{5}$ of type B in a subgroup  $\dih{10}$, as this simplifies the proof of   \thm{no_l2_59}.
    \end{remark}
    
  
    \section{\label{psgl}The Projective Linear Maximal Subgroups of the Monster}

    \noindent We discuss projective linear groups separately because confirmation of their construction is much easier than in most other cases due to the availability of presentations. Specifically, in the code we use Sunday's presentation \cite{sl_pres_1} for $\lins{2}{q}$ where $q$ an odd prime power, and Robertson and Williams' presentation \cite[Theorem~4]{pgl_pres} for $\pgl{2}{p}$ with $p$ a prime, with the correction of  Hert   \cite[Theorem 3]{pgl_pres_fixed} who pointed out that the presentation for $\pgl{2}{p}$ in \cite[Theorem~A]{pgl_pres} is in fact
    a presentation for $2 \times \psl{2}{p}$ whenever $2$ is
    a quadratic residue modulo $p$. A common step in each of the following proofs is to verify a presentation; to keep the exposition short, we do not reiterate this every time, and only provide additional information on maximality if necessary. 

    \sg{$\psl{2}{71}$}{l2_71}{
        The elements $g_{71}, h_{71}$ in the accompanying code generate
        a maximal subgroup  $\psl{2}{71}$.
    }
    
    \begin{proof}
       It follows from 
        \cite[Theorem~1, \S 5.4]{subgroups_A5} that  all $\psl{2}{71} < \MM$ are
        conjugate and maximal.
    \end{proof}

    \sg{$\psl{2}{41}$}{l2_41}{
        The elements $g_{41}, h_{41}$ in the accompanying code generate  a maximal subgroup  $\psl{2}{41}$.
    }
    
    \begin{proof}
       It follows from \cite[Theorem~1]{l2_41} that all $\psl{2}{41} < \MM$ are conjugate and maximal.
    \end{proof}

    \sg{$\pgl{2}{29}$}{pgl2_29}{
        The elements $g_{29}, h_{29}$  in the accompanying code generate  a maximal subgroup 
        $\pgl{2}{29}$.
    }
    
    \begin{proof}
     It follows from \cite[Theorem~1]{pgl2_29} that all $\pgl{2}{29} < \MM$ are conjugate and maximal.
    \end{proof}

    \sgNEW{$\pgl{2}{19}$}{pgl2_19}{
        The elements $x_2, x_{19}$  in the accompanying code generate  a maximal subgroup 
        $\pgl{2}{19}$.
    }{This theorem is adapted from \cite[\S6]{four_fus}.}

    \begin{proof}
        It follows from  \cite[Theorem~1]{subgroups_A5}   that $\langle x_2, x_{19} \rangle \cong \psl{2}{19}$ is maximal in $\MM$
        if and only if it contains an $\alt{5}$ of type B; the claim follows since         $( x_{19}^2 x_2 )^2$ and $x_2 x_{19}^2 x_2 x_{19}$ are
        the generators $g_{2B}$ and $g_{3B}$ of the $\alt{5}$ of type B 
        in \lem{a5_lemma}.
    \end{proof}

    \sgNEW{$\pgl{2}{13}$}{pgl2_13}{
        The elements $u, g_{13} $  in the accompanying code generate  a maximal subgroup
        $\pgl{2}{13}$.
    }{This theorem is adapted from  \cite[\S3]{dlp}.}

    \begin{proof}
        To show that this subgroup
        is maximal, we first note that there are known to be exactly
        three conjugacy classes of $\psl{2}{13}$ in $\MM$. Two of these,
        identified by Norton \cite[\S5]{anatomy}, have centralisers of shapes
        $3^{1+2}.2^2$ and $3$; the third, found by Dietrich \etal \cite{dlp},
        has trivial centraliser and extends to a maximal subgroup $\pgl{2}{13}$.
        It thus suffices to exhibit $\psl{2}{13} < \langle u, g_{13} \rangle$
        which is not centralised by an element of order $3$.  As such, let $x_2 = ( ug_{13}^2 )^2 \in
        \langle u, g_{13} \rangle$; we verify  $\langle x_2, g_{13} \rangle \cong \psl{2}{13}$. It is also clear
        that $\cma{x_2, g_{13}} \le \cm{g_{13}}$, One may recall
        from the proof of \thm{norm13A} that the latter is embedded in the direct product
        $13{:}6 \times \psl{2}{13}$ of $\langle g_{13}, y_{12}^2 \rangle \cong
            13{:}6$ and its centraliser. Writing
        $g_6 = g_{13}^{-1} x_2 g_{13}^7 x_2 g_{13}^2 x_2 \in
            \langle x_2, g_{13} \rangle$, the fact that $x_6 = y_{12}^{-2} g_6$
        commutes with both $g_{13}$ and $y_{12}^2$ then reveals that
        $g_6$ is the product of $y_{12}^2 \in 13:6$ and $x_6 \in \psl{2}{13}$. It follows that
        any element of $13{:}6 \times \psl{2}{13}$ which commutes with $g_6$
        must also centralise both $y_{12}^2$ and $x_6$. Hence
        $\cma{g_{13}, x_2} \le \cma{g_{13}, g_6} \le \cma{g_{13}, y_{12}^2, x_6}
            = C_{\psl{2}{13}} ( x_6 )$, with $x_6$ an element of
        order $6$ in  $\cma{g_{13}, y_{12}^2} \cong \psl{2}{13}$. The Atlas implies that this centraliser is  $\langle x_6 \rangle$, and we show that $x_6^2$ (of order 3) does not commute with $x_2$. Thus, no element of order $3$ centralises $\langle x_2, g_{13} \rangle$.
    \end{proof}

    \subsection{\label{correction}A correction for $\psl{2}{59}$}\label{sec_psl259}
    
    Attempting to reproduce the methodology of Holmes and Wilson's \cite{l2_59} construction of a maximal $\psl{2}{59}$, we come to the conclusion that no such subgroup exists. Our argument is spelled out in the following proof; it reveals that there is instead a new maximal subgroup of $\MM$, which is isomorphic to $59{:}29$.
    
    \begin{theorem}
        \label{thm_no_l2_59}
        There is no subgroup $\psl{2}{59}$ of $\MM$.
        Consequently, the normalisers $59{:}29$ of elements of
        order $59$ form a class of maximal subgroups of $\MM$.
    \end{theorem}
    
    \begin{proof}
      The group $\psl{2}{59}$ has maximal subgroup $\alt{5}$ and every  element of order $5$
        has a normaliser $\dih{60} \cong 2 \times \dih{30}$.  Norton and Wilson \cite{anatomy2} have shown that elements of order $2$, $3$, or $5$ in a maximal $G=\psl{2}{59}$ in $\MM$ must lie in $\MM$-classes 2B, 3B, or 5B,
        respectively. It follows from Table~\ref{a5_classes} that every subgroup
        $\alt{5}$ of $G$ is of type  B or T. Let
        $g_{2B}$, $g_{3B}$, $g_{2T}$, and $g_{3T}$ be as in \lem{a5_lemma},
        so that  $\langle g_{2B}, g_{3B} \rangle$
        and $\langle g_{2T}, g_{3T} \rangle$ are subgroups isomorphic to $\alt{5}$s of the required types. In the code we fix an element $c$ such that the putative $G$ contains
        at least one of $\langle g_{2B}, g_{3B} \rangle^c$
        and $\langle g_{2T}, g_{3T} \rangle$, see also Remark \ref{good_a5}. Let $b_5 = g_{2T} g_{3T}$; a computation reveals that 
        $g_{2T} g_{3T} = b_5 = g_{2B}^c g_{3B}^c$ and
        $g_{2T}^{g_{3T} g_{2T} g_{3T}^2} = (
            g_{2B}^{g_{3B} g_{2B} g_{3B}^2}
        )^c = z$, where $z$ is the central involution in $\GG$.
        Our assumed $G \cong \psl{2}{59}$ therefore contains $\langle b_5, z \rangle$, which we confirm to be $\dih{10}$. On the other hand, $\dih{10}$ is
        a group with trivial centre, normalising elements of order $5$ in $G$; since
        such elements have normalisers $2 \times \dih{30}$ in $G$, the group
        $\langle b_5, z \rangle$ must be centralised by an involution
        $x \in G$; in particular, we must have $x \in \textup{2B}$ as shown above.  We check that no suitable $x$ exists, and for this it suffices to consider all 2B
        involutions in the centraliser of the $\dih{10}$ we found above. For these subgroups $\dih{10}$ in $\alt{5}$ of type B or T,  Norton \cite[Table~4]{anatomy} proves
        that the relevant centralisers have shape $5^3.( 4 \times \alt{5} )$; thus,  any list of $| 5^3.( 4 \times \alt{5} ) | = 30000$ distinct elements commuting with $b_5$ and $z$ generate the whole centraliser. In the code, we generate
        such a list from products of $6$ pre-computed generators. Testing orders and $\chi_\MM$ values  yields
        that exactly $500$ involutions of type 2B remain. This coincides with the computations reported by Holmes and Wilson \cite[p.~13]{l2_59}, and gives confidence that our construction so far is in line with the one in \cite{l2_59}. In \cite{l2_59} it is claimed that among these $500$ involutions, some elements $x$ extend $\alt{5}$ to a subgroup $\psl{2}{59}$. However, we cannot confirm this:  note that elements in $\psl{2}{59}$ have orders
        in $\mathcal{O} = \{ 1, 2, 3, 5, 6, 10, 15, 29, 30, 59 \}$.
        In particular, if $G$ contains a subgroup $\alt{5}$ of type B, then all of $x ( g_{2B}^{g_{3B} g_{2B}} )^c$, $xg_{3B}^c$,
            and $g_{2B}^c x g_{3B}^c$ have orders in $\mathcal{O}$
        for some involution $x$ found above. Similarly, if $G$ contains a subgroup $\alt{5}$ of type T, then the orders of $xg_{2T}^{g_{3T} g_{2T}}$ and
        $g_{2T} xg_{3T}$ must belong to $\mathcal{O}$ for some such $x$. We run over all $500$ involutions and conclude that none of them satisfies these conditions. This contradiction allows us to conclude that $\MM$ has no subgroup $G=\psl{2}{59}$.

        The existence of a subgroup $H=59{:}29$ of $\MM$ follows from \cite{odd_local}. It remains to prove maximality.  If $H$ is not maximal, then $H<M$ for some maximal $M<\MM$. Note that $M=N_\MM(T)$ where $T=S^m$ is a direct product of isomorphic simple groups $T$. 
   If $59$ divides $|T|$, then $m=1$ and so $M=N_\MM(59) = 59{:}29$ by \cite{odd_local}, contradicting our assumption that $H$ is not maximal.  If $59$ does not divide $|T|$, then the normal $59<H$ normalises $T=S^m$.  Since no prime power $p^{59}$ divides $|\MM|$, we deduce that $m<59$ and therefore $59<H$ does not permute the $m$ factors of $T$. Thus this $59$ must normalise each $S$, so induces an element of order $59$ in $\Aut(S)/S$. Running over all simple subgroups $S$ that could be involved in $\MM$, we determine that $59$ never divides $|\Aut(S)/S|$. Thus, this case is also not possible. Hence, $H$ is maximal, as claimed.
    \end{proof}
    
\begin{remark}\label{remget59}
Establishing elements in \mmgroup{} that generate a subgroup $59{:}29$ constitutes ongoing work. While elements of orders $59$ and $29$  may be easily found by random search,  constructing the normaliser of a given element of order $59$ is very difficult -- a na\"ive approach yields an approximately $1$ in $10^{48}$ chance of success.  An adaptation of a method developed by Bray et al.\ \cite{find_47_23} for an analogous problem for the Baby Monster increases the probability of success to around 1 in $10^8$. 
\end{remark}

    
    \section{\label{nonlocal}Non-Local Maximal Subgroups of the Monster}

\noindent We now consider the non-local maximal subgroups of $\MM$. 
    
    \begin{procedure}\label{nonlocalproc}
    For many non-local subgroups of $\MM$, it is convenient to use
    a modification of Procedure~\ref{provesg} when performing verifications; we usually proceed as follows.
    
    \begin{enumerate}
        \item We exhibit generators $L_E$ for 
            a subgroup $E$ such that $\nm{E}$ is a maximal subgroup of $\MM$ of
            the claimed structure; the results of Section \ref{a5s}, \cite[p. 234]{atlas} and
            \cite[\S4, \S5]{anatomy} often provide the necessary data.
        \item We propose generators $L_N$ for $\nm{E}$ and first ascertain that $L_N$ normalises $E$. Some generators in $L_N$ can be written as a product $gh$ where $g \in \cm{E}$ and $h \in E$; for others, it may be
            directly verified that they map $L_E$
            to elements of $E$ under conjugation.
        \item The remaining steps are then to confirm that $L_N$ generates $\nm{E}$. We usually begin
            by exhibiting words for the generators of $E$ in elements of $L_N$, so $E\leq \langle L_N\rangle$, and then show that $C_\MM(E)\leq \langle L_N\rangle$.  The elements of $\cm{E}$ used to write
            products in Step (2) belong to $\langle L_N\rangle$  since those of $E$ do,
            so that it frequently suffices to show these elements generate $\cm{E}$.
        \item Finally, we demonstrate that $L_N$  extends
            $E \cm{E}$ to $\nm{E}$, usually by exhibiting elements
            which induce suitable outer automorphisms of $E$.
    \end{enumerate}
    \end{procedure}

\noindent  Throughout the following, we refer to the elements listed in \lem{a5_lemma}.

    \sg{$\normAAA$}{normAAA}{
        Let $L_E=\{g_{2AAA}, g_{3AAA}\}$ and $L_N=\{g_{AAA}, h_{AAA}, n\}$ be as 
        in the accompanying code. Then $L_N$ generates
        the maximal subgroup $N_\MM(E) \cong \normAAA$.
    }

    \begin{proof}
        According to the Atlas,  $N_\MM(L_E)\cong \normAAA$, and it remains to prove that $\langle L_N\rangle=N_\MM(E)$, see also \lem{a5_lemma}. We verify  that 
         $x_3 = g_{AAA} g_{2AAA}$ and $x_{10} = h_{AAA} g_{3AAA}^{-1}$ centralise $E$, and confirm that  $n$  centralises
        $g_{2AAA}$ and that $g_{3AAA}^n\in E$. The latter is checked via an enumeration of $E$, which also confirms that conjugation by $n$ induces an outer automorphism of $\alt{5}$.       We verify  $g_{2AAA} = g_{AAA}^3$ and $g_{3AAA} = h_{AAA}^{10}$, so $E$ and $x_3,x_{10}$ lie in $\langle L_N\rangle$. We show that
        $\langle x_3, x_{10} \rangle$ is the full centraliser $\alt{12}$ of $E$ by verifying that the generators $x_3$, $x_{10}$ satisfy the group presentation of $\alt{12}$ obtained by   Coxeter and Moser
        \cite[\S6.4]{cox_mos}. Since $n$ induces an outer automorphism,   $\langle L_N\rangle\cong \normAAA$.
    \end{proof}
 
A subgroup $\alt{5} \times \alt{12}$ has already been constructed in \cite{dlp}; in our code we provide a conjugating element that maps $\alt{5}$ into $\GG$.

\sg{$\normASixCu$}{normA6_3}{
        Let $L_U=\{x_{A6},y_{A6}\}$, $t$, and  $L_N=\{g_{A6}, h_{A6}\}$ as in
        in accompanying code. Then $L_N$ generates the maximal subgroup $\nma{E} \cong \normASixCu$ where $E=\langle U,U^t,U^{t^2}\rangle$ with $U=\langle L_U\rangle\cong \alt{6}$.
    }
    
\begin{proof}
     
        The Atlas shows that the maximal subgroups $\normASixCu$ are precisely the normalisers of \ $\alt{6} \times \alt{6} \times \alt{6}$, where each $\alt{6}$ contains 2A, 3A, 4B, and 5A elements. We verify that this holds for 
       the subgroup $U = \langle x_{A6}, y_{A6} \rangle \cong \alt{6}$ using the presentation given in \cite[\S6.4]{cox_mos}. Moreover, $U$ contains elements of a subgroup $\alt{5}$ of type AAA since it contains
        $g_{2AAA} = x_{A6} y_{A6} x_{A6} y_{A6}^2 x_{A6}^2$ and
        $g_{3AAA} = x_{A6} y_{A6}^2 x_{A6}^2 y_{A6} x_{A6}$ as in Lemma \ref{a5_lemma}.  It follows from \cite[Table~5]{anatomy} that 
        there is a unique conjugacy class of groups $\alt{6} < \MM$ containing
        such elements. This shows that  $U, U^t, U^{t^2}$ are of the required type; we verify that they commute and that $t$ has order $3$. An explicit calculation shows that $g_{A6}$ and $h_{A6}$ permute the groups $U$, $U^t$, $U^{t^2}$ cyclically, and we verify $E\leq \langle L_n\rangle\leq N_\MM(E)$ by writing the generators $L_U$ as words in elements of $L_N$, and confirming that $g_{A6}^4 \in \langle L_N\rangle$. It remains to show that $\langle L_N\rangle$ contains the factor $2 \times \sym{4}$. The element  $c_0 = ( g_{A6}^3 h_{A6}^2 g_{A6} h_{A6}^3)^6$
        centralises $x_{A6}, x_{A6}^t$ and $y_{A6}^t$, but $y_{A6}^{c_0} \notin y_{A6}^{\langle x_{A6} \rangle}$. Any element of $E$ which
        centralises $x_{A6}$ is  the product of an element
        centralising $U$ and an element of $C_U ( x_{A6} ) =
            \langle x_{A6} \rangle$, so we deduce that    $c_0 \notin E$, so  $\langle E, c_0 \rangle$        is at least twice as large as $E$ by \lem{ext_lemma}.
        The element $c_1 =
            ( h_{A6}^2 g_{A6}^3 h_{A6}^3)^6$ centralises
        $U$ and induces an outer automorphism of $U^t$ which
        adjoining $c_0 \in \cm{U^t}$ cannot produce; thus, $\langle E, c_0, c_1 \rangle$ is
        at least $2^2 = 4$ times as large as $E$.   
        We next check that  $U^t$, $U^{t^2}$, and $c_1$
        centralise $U$; on the other hand,
        since $c_0 \notin U$ is an involution normalising $U$,
        the Isomorphism Theorem shows $\langle x_{A6}, y_{A6}, c_0 \rangle/U
            \cong \langle c_0, U \rangle/U \cong
            \langle c_0 \rangle/( \langle c_0 \rangle \cap U ) =
            \langle c_0 \rangle$, so that all automorphisms of $U$
        induced by the group generated so far arise from the product
        of an element of $U$ and $\langle c_0 \rangle$. Of these,
        $4 \cdot 2 = 8$ elements centralise $x_{A6}$.
        Establishing that the conjugate of $y_{A6}$ by $c_2 = g_{A6}^3 y_{A6}^{-1} \in
            \langle g_{A6}, h_{A6} \rangle$ is not among the $8$ elements in the $\langle x_{A6},c_0\rangle$-class of $y_{A6}$ therefore suffices to show that
        $\langle E, c_0, c_1, c_2 \rangle$ is at least $2^3 = 8$
        times as large as $E$. Moreover, recalling that conjugation by elements
        of $\langle g_{A6}, h_{A6} \rangle$ permutes the factors of $E$,
        checking that $U^{tc_2}$ commutes with $U^{t^2} \cong \alt{6}$ and normalises $U$ 
        shows that $U^{tc_2} = U^t$ and hence that $c_2$,
        like all other generators of the extended group just exhibited,
        normalises $U, U^t$ and $U^{t^2}$. Adjoining an element
        $c_3=h_{A6}^{t^{-1}}$ which interchanges $U$ and $U^t$ thus produces a group
        at least $2 \cdot 2^3 = 16$ times the size of
        $E$.
        Finally, since $t$ does not normalise
        $U^{t^2}$, the automorphism of $U$ induced by $t$ is still unaccounted for after the introduction
        of $c_3$. Thus, \[ \langle E, c_0, c_1, c_2, h_{A6}^{t^{-1}}, t \rangle
            \le \langle g_{A6}, h_{A6} \rangle \] has order at least
        $3 \cdot 2^4 | \alt{6} |^3 = | \normASixCu |$, and the claim follows.
    \end{proof}

\sg{$\normACA$}{normACA}{
  Let $L_N=\{g_{3ACA},g_{ACA}, h_{ACA}\}$ and $L_E=\{g_{2ACA},g_{3ACA}\}$ be as in the accompanying code. Then $\langle L_E\rangle \cong \alt{5}$ and $L_N$ generates the maximal subgroup  $N_\MM(E)\cong \normACA$.
    }
    
\begin{proof}
  The group $E$ is a subgroup $\alt{5}$ of type ACA, and it follows from  \cite[p.~234]{atlas} that $N_\MM(E)\cong\normACA$ is maximal in $\MM$. We verify that $L_E\subseteq \langle L_N\rangle$ and that  $v = g_{ACA} g_{2ACA}$
        centralises $E$.  We have $g_{2ACA} = g_{ACA}^3$, so $v\in \langle L_N \rangle$. We claim that the centraliser  $\unt{3}{8}{:}3$ of $E$ lies in $\langle L_N\rangle$ and is generated by $h_{ACA} v h_{ACA} v^2$,
        $( h_{ACA} v )^3 ( h_{ACA} v^2 )^3$, and
        $( h_{ACA} v )^3 h_{ACA} v^2$. These elements commute
        with $L_E$ and have orders  $19$, $7$, and $12$, respectively. The former two must lie in $\unt{3}{8} < \unt{3}{8}{:}3$. It follows from the Atlas that  $\unt{3}{8}$  has no maximal subgroups of order divisible by $7$ and $19$,
        which combined with the fact that  $\unt{3}{8}{:}3/\unt{3}{8} \cong 3$
        is cyclic of prime order and $\unt{3}{8}$ contains no elements of order
        $12$  establishes that the three elements  indeed
        generate the whole of $C_\MM(E)$. Finally, we confirm (analogously to the proof of \thm{normAAA}) that  conjugation by $h_{ACA}$ induces an outer automorphism  of $\alt{5}$; thus, $\langle L_N\rangle\cong \normACA$.
    \end{proof}

\sg{$\normLThreeTwo$}{normL7}{
Let $L_E=\{ x_{L2(7)}, y_{L2(7)}\}$ and $L_E=\{g_{L2(7)}, h_{L2(7)},n\}$ be as in the accompanying code. Then $E\cong \psl{2}{7}$ and $L_N$ generates the maximal subgroup $N_\MM(E)\cong\normLThreeTwo$.
    }
    
    \begin{proof}
        By \cite[\S5, Case 1]{anatomy}, the maximal subgroup we aim to construct is a normaliser of a subgroup $\psl{3}{2}$ with elements in  $\MM$-classes 2A, 3A, and 7A. We first check a presentation and verify that $E\cong \psl{2}{7}$; note that the latter is also isomorphic to $\psl{3}{2}$ by \cite[Prop. 2.9.1, (xi)]{highdeg_class}. Our generators in $L_E$ lie in $\GG$ and we verify that  $x_{L2(7)}, y_{L2(7)}, x_{L2(7)} y_{L2(7)} \in \GG$ lie in classes 7A, 2A, and 3A, respectively. We establish
        $g_{L2(7)}, h_{L2(7)}, n \in N_\MM(E)$ by determining that  $u = g_{L2(7)} x_{L2(7)}^{-1}$ and 
        $v = h_{L2(7)} y_{L2(7)}$ centralise $L_E$, while $n$ inverts $x_{L2(7)}$ and commutes with $y_{L2(7)}$. We have  $x_{L2(7)} = g_{L2(7)}^8$ and $y_{L2(7)} = h_{L2(7)}^{17}$, so $E\leq \langle L_N\rangle$. We know that $u$ and $v$ lie in $C_\MM(E)\cong \symp{4}{4}{:}2$, and show that they actually generate the centraliser. Since  squares of all elements
        in $\symp{4}{4}{:}2$ lie in the normal subgroup $\symp{4}{4}$,
        we conclude that the latter's intersection with $\langle u, v \rangle$
        contains the element $v^2$ of order $17$. Atlas information implies that the only maximal subgroup
        of $\symp{4}{4}$ containing such an element is $\psl{2}{16}.2$.
        On the other hand, the orders of involution centralisers in $\psl{2}{16}.2$
        are not divisible by any power of $2$ greater than $16$; it will thus
        follow that $\langle u, v \rangle$ contains $\symp{4}{4}$ if it can
        be shown that it contains an involution centraliser of
        order a multiple of $32$. To this end,  we construct five elements as word in $u,v$ that commute with the involution 
         $(vuv )^2$,  and generate a group of order $64$ as sought. To complete the proof, we must verify that $\langle u, v \rangle$
        properly contains $\symp{4}{4}$, so that $\langle u,v\rangle=C_\MM(E)$, and that $n$ extends $\psl{3}{2}\times \symp{4}{4}{:}2$ to the full normaliser of $E$. The first assertion is demonstrated by an adaptation of above argument: the largest $2$-power that divides the order
        of an involution centraliser in $\symp{4}{4}$ is $256$, whereas the elements above together with
        with $ ( w vw u )^2 w, ( v^2 w )^2 v $ generate a group of order $512$ centralising $( vuv )^2$. For the second result, an  additional factor of $2$ given
        by an outer automorphism of $E$ is required.
        Since elements of order $7$ in $\psl{3}{2}$ have normalisers $7{:}3$, if follows that  $n$ (which inverts $x_{L2(7)}$) induces        the required extension.
    \end{proof}

    \sg{$\normLTwoEleven$}{normL11}{
      Let $L_E=\{g_{2AAA},x_{11}\}$ and $L_N=\{g_{11}, h_{11}, n\}$ as in the accompanying code.
       Then $E\cong \psl{2}{11}$ and $L_N$ generates the maximal subgroup $N_\MM(E)\cong  \normLTwoEleven$.}
    
    \begin{proof}
        By the Atlas, maximal subgroups
        $\normLTwoEleven$  are the normalisers of subgroups isomorphic to $\psl{2}{11}$ containing elements of the  $\MM$-classes
        2A, 3A, and 5A. We confirm a presentation and conjugacy class fusion for $E$, so $N_\MM(E)\cong \normLTwoEleven$; specifically, we show that $E$ contains a subgroup $\alt{5}$ of type AAA generated by $g_{2AAA}$ and $g_{3AAA} =
        g_{2AAA} x_{11} g_{2AAA} x_{11}^3$. We also confirm that $u = g_{11} g_{2AAA}$ and
        $v = h_{11} x_{11}^{-1}$ both centralise $g_{2AAA}$ and $x_{11}$ (note that   $g_{2AAA}$ and $x_{11}$ clearly both belong to the normaliser) while
        $n$ commutes with the former and inverts the latter.  We compute that $g_{2AAA} = g_{11}^3$ and
        $x_{11} = h_{11}^2$, so $E\leq \langle L_N\rangle$. Next, we use Atlas information to deduce that $u$ and $v$ generate the subgroup $\cma{g_{2AAA}, x_{11}} \cong \MT{12}$: the latter has three maximal subgroups of order divisible by $11$, but none of these has elements of order $10$, whereas $vu$ and $( vu )^4 uvu^2$ have orders $11$ and $10$, respectively. Thus, $E\times C_\MM(E)\leq \langle L_N\rangle$. Elements of order $11$ in $\psl{2}{11}$ have normalisers $11{:}5$, whereas $n$ inverts $x_{11}$. This proves that $\langle L_N\rangle\cong \normLTwoEleven$.
    \end{proof}

    
    \sg{$\normASeven$}{normA7}{
     Let $L_U=\{ g_{2AAA}, g_{3AAA}\}$, $\sigma$, and  $L_N=\{g_{A7}, h_{A7}\}$ as in
        in accompanying code. Then $L_N$ generates the maximal subgroup $N_\MM(E) \cong \normASeven$ where $E=\langle U,U^\sigma\rangle$ with $U=\langle L_U\rangle\cong \alt{5}$.
    }
    
    \begin{proof}
      By the Atlas, normalisers of the direct product of two $\alt{5}$s of type AAA are exactly the maximal subgroups $\normASeven$. We verify that $E\cong \alt{5}\times\alt{5}$ has the required type and that $L_N\subset N_\MM(E)\cong \normASeven$. 
       We confirm that $E\leq \langle L_N\rangle$ by showing that $\sigma$ and the elements of $L_U$ can be written as words in the elements of $L_N$.
       Moreover, we verify that 
      $x = ( h_{A7} g_{A7} h_{A7} g_{A7}^2 )^6$
        and $y = ( g_{A7} h_{A7}^3 )^4$ commute
        with the generators of $E$ and satisfy the presentation for $\alt{7}$ given in  \cite[\S6.4]{cox_mos}; thus, $\langle L_N\rangle$ contains  $P \cong \alt{7} \times \alt{5} \times \alt{5}$. We now show that  $a = ( h_{A7} g_{A7}^6  )^{15}$ commutes with $g_{2AAA}$, $g_{2AAA}^{\sigma}$ and $g_{3AAA}^{\sigma}$,
        while $g_{3AAA}^a \in U$. On the other hand, $a$ cannot belong
        to $P$: since all elements
        thereof which centralise $g_{2AAA}$ are the product of an element
        of $C_U ( g_{2AAA} )$ and an element centralising $U$,
        it would otherwise hold that $g_{3AAA}^a$ lies in the $C_U(g_{2AAA})$-class of $g_{3AAA}$, which we disprove by an enumeration.  Thus, $P$ has index at least $2$ in $\langle P,a\rangle\leq \langle L_N\rangle$. The analogous argument with $a^\sigma$ yields a subgroup of $\langle L_N\rangle$ in which $P$ has index at least $4$. All generators so far  normalise $U$ and $U^{\sigma}$, so noting that $\sigma$ swaps these factors produces a subgroup of $\langle L_N\rangle$ in which $P$ has index at least $8$. This proves that $\langle L_N\rangle=N_\MM(E)$.
    \end{proof}

\pagebreak

    \sg{$\normASix$}{normA6}{

      Let $L_E=\{g_{2AAA},y_4\}$ and $L_N=\{g_{M11}, h_{M11}\}$ be as in the accompanying code. Then $L_N$ generates the maximal subgroup  $N_\MM(E)\cong  \normASix$ of $\MM$.
    }
    
    \begin{proof}
       By the  Atlas, the maximal subgroup $\normASix$ is the normaliser of a subgroup $\MT{11}$ containing 2A, 3A, and 5A elements. We check a presentation to prove $E\cong \MT{11}$. The group $E$ contains a subgroup $\alt{5}$ of type AAA meeting the correct conjugacy classes since the relevant generators defined in Lemma \ref{a5_lemma} can be written as words in elements of $L_E$.
%
         Since $u = g_{M11} y_4^{-1} g_{2AAA}$ and
        $v = h_{M11} y_4^{-1} (g_{2AAA} y_4 )^{-3}$ centralise $E$, it follows that $L_N\leq N_\MM(E)$, and  $y_4 = g_{M11}^{30} h_{M11}^{-10}$ and
        $g_{2AAA} = g_{M11}^{-10} y_4^{-1}$ confirms that $E\leq\langle L_N\rangle$. Lastly, we confirm that $\langle v,u\rangle$ has size $|\alt{5}\udot 2^2|$, and therefore $\langle L_N\rangle=N_\MM(E)$.
    \end{proof}

    \sg{$\normSFiveCu$}{normS5_3}{
 Let $L_U=\{  x_{S5}, y_{S5}\}$, $t$, and  $L_N=\{g_{S5}, h_{S5}\}$ as in
 in accompanying code. Then $L_N$ generates the maximal subgroup $N_\MM(E) \cong \normSFiveCu$ where $E=\langle U,U^t,U^{t^2}\rangle$ with $U=\langle L_U\rangle\cong \sym{5}$.
    }
    
    \begin{proof}
        By the Atlas, maximal subgroups $\normSFiveCu$ of $\MM$ are
        the normalisers of direct products of three subgroups $\sym{5}$
        containing 2A, 3A, and 5A elements. We test that $U$ has order $120$ and that $L_E$ satisfies a presentation for $\sym{5}$, so $U\cong \sym{5}$.   The fact that
        $g_{2AAA} = ( x_{S5} y_{S5}^2 )^2 x_{S5} y_{S5} x_{S5}$
        and $g_{3AAA} = x_{S5} y_{S5} x_{S5} y_{S5}^3$ guarantees that $U$ meets the correct conjugacy classes. We verify that $U$, $U^t$, and $U^{t^2}$ commute and intersect trivially (since the factors have trivial centre), so $E\cong \sym{5}\times\sym{5}\times\sym{5}$, as required. A direct computation also shows that $L_N\leq N_\MM(E)$, and we verify $E\leq \langle L_N\rangle$ by writing the generators in words of elements of $L_N$.
        An explicit computation also shows that  $h_{S5}$ centralises $U = \langle x_{S5}, y_{S5} \rangle$ and interchanges
        $U^t, U^{t^2}$ via conjugation, while conjugation by $t$ permutes
        $\{ U, U^t, U^{t^2} \}$ cyclically. This allows us to deduce that $\langle L_N\rangle \cong \normSFiveCu$.
    \end{proof}

    \sgNEW{$\normLTwoElevenSq$}{normL11_2}{
      Let $L_U=\{g_{2AAA},x_{11}\}$ and $L_N=\{g_{2AAA}, x_{11}, x_4\}$ be as in the accompanying code. Then $L_N$ generates a maximal subgroup $N_\MM(E)=\normLTwoElevenSq$ where $E=\langle U,U^{x_4}\rangle$ with $U=\langle L_U\rangle\cong \psl{2}{11}$.
    }{This is adapted from  \cite[\S3]{four_fus}.}

    \begin{proof}
        According to the Atlas, the maximal subgroups
        $\normLTwoElevenSq$ of $\MM$ are the normalisers of direct products
        $\psl{2}{11} \times \psl{2}{11}$ in which the elements of order
        2, 3, and 5 in each factor belong to classes 2A, 3A, and 5A.
        It was shown in the proof of \thm{normL11} that $U\cong \psl{2}{11}$ meets the correct conjugacy classes. We verify that $U$ and $U^{x_{4}}$ commute, so $E=\psl{2}{11}\times\psl{2}{11}$ and $N_\MM(E)\cong \normLTwoElevenSq$. 
        Since $x_4^2$ centralises  $g_{2AAA}$ and inverts $x_{11}$, we deduce that $L_N\subset N_\MM(E)$. As noted in the proof of \thm{normL11},  every automorphism of $U$ that inverts $x_{11}$ is
        an outer automorphism. Therefore $x_4^2\notin E$, and the claim follows.
    \end{proof}

    \sgNEW{$\unt{3}{4}{:}4$}{u34}{
      Let $L_N=\{ j_2,a_{12},g_{3BCB}\}$ be as in the accompanying code. Then $L_N$ generates a maximal subgroup 
        $\unt{3}{4}{:}4$.
    }{This is adapted from \cite[\S6]{dlp}.}

    \begin{proof}
        Per \cite[\S6]{dlp}, the maximal subgroups of $\MM$
        with the desired shape are the normalisers of subgroups $\unt{3}{4}$
        containing $\alt{5}$s of type BCB (as opposed to a second class,
        identified by Norton \cite[Table~5]{anatomy}, that contain subgroups $\alt{5}$
        of type BCA). We verify that $E=\langle j_2,g_{3BCB}\rangle\cong \unt{3}{4}$ by verifying a presentation. The presence of a subgroup $\langle g_{2BCB},g_{3BCB}\rangle\cong \alt{5}$ of type BCB is established by verifying that we can write the relevant generators in Lemma \ref{a5_lemma} as words in elements of $L_N$.
            An explicit calculation in the accompanying code confirm that $L_N\leq N_\MM(E)$. Now the claim follows since $a_{12}$ has order $12$ but $\unt{3}{4}$ has no elements of order $12$ or $6$. 
    \end{proof}




\begin{small}

\end{small}


\begin{thebibliography}{10}



   




\bibitem{find_47_23}
J.~N. Bray, R.~A. Parker, and R.~A. Wilson.
\newblock Finding $47{:}23$ in the Baby Monster.
\newblock {\em London Math.\ Soc.\ J.\ Comput.\ Math.}, 19:229--234, 2016.


\bibitem{gap_ctbllib}
T.~Breuer.
\newblock {\em \texttt{CTblLib} -- a GAP package}.
  {\url{math.rwth-aachen.de/homes/Thomas.Breuer/ctbllib/}}









\bibitem{atlas}
J.~H. Conway, R.~T. Curtis, S.~P. Norton, R.~A. Parker, and R.~A. Wilson.
\newblock {\em Atlas of Finite Groups}.
\newblock Oxford University Press~1985.


\bibitem{cox_mos}
H.~S.~M. Coxeter and W.~O.~J. Moser.
\newblock {\em Generators and relations for discrete groups}.
\newblock Springer, 3rd edition, 1972.






\bibitem{dlp}
H.~Dietrich, M.~Lee, and T.~Popiel.
\newblock The Maximal Subgroups of the Monster.
\newblock preprint, 2023.
{\url{arxiv.org/abs/2304.14646}}

\bibitem{gap}
The GAP Group.
\newblock {\em GAP --- Groups, Algorithms and Programming}.
 {\url{https://gap-system.org}}





\bibitem{pgl_pres_fixed}
T.~M. Hert.
\newblock An Efficient Presentation of $\pgl{2}{p}$.
\newblock Master's thesis, California State University, San Bernadino, 1993.




  
\bibitem{pgl2_29}
P.~E. Holmes and R.~A. Wilson.
\newblock A New Maximal Subgroup of the Monster.
\newblock {\em J.\ Algebra}, 251:435--447, 2002.


\bibitem{monp}
P.~E. Holmes and R.~A. Wilson.
\newblock A new computer construction of the Monster using 2-local subgroups.
\newblock {\em J.\ London Math.\ Soc.}, 67:349--364, 2003.


\bibitem{l2_59}
P.~E. Holmes and R.~A. Wilson.
\newblock $\psl{2}{59}$ is a subgroup of the Monster.
\newblock {\em  J.\ London Math.\ Soc.}, 69:141--152, 2004.


\bibitem{subgroups_A5}
P.~E. Holmes and R.~A. Wilson.
\newblock On subgroups of the Monster containing $A_5$'s.
\newblock {\em J.\ Algebra}, 319:2653--2667, 2008.


\bibitem{hcgt}
D.~F. Holt, B.~Eick, and E.~A. O’Brien.
\newblock {\em Handbook of Computational Group Theory}.
\newblock Discrete Mathematics and its Applications. Chapman \& Hall/CRC, FL,
  2005.








\bibitem{highdeg_class}
P.~B. Kleidman and M.~W. Liebeck.
\newblock {\em The Subgroup Structure of the Finite Classical Groups}, volume
  129 of {\em London Math.\ Soc.\ Lect.\ Note Ser.}.
\newblock Cambridge University Press, 1990.






\bibitem{2local2}
  U.\ Meierfrankenfeld.
  The maximal 2-local subgroups of the Monster and Baby Monster, II.
preprint, \url{https://users.math.msu.edu/users/meierfra/Preprints/2monster/2MNC.pdf}, 2003.

   




\bibitem{anatomy}
S.~P. Norton.
\newblock Anatomy of the Monster: I.
\newblock In {\em The atlas of ﬁnite groups: ten years on}, volume 249 of
  {\em London Math.\ Soc.\ Lect.\ Note Ser.}, pages 198--214, Cambridge Univ. Press, 1998.
  


\bibitem{anatomy2}
S.~P. Norton and R.~A. Wilson.
\newblock Anatomy of the Monster: II.
\newblock {\em Proc.\  London Math.\ Soc.}, 84:581--598,
  2002.


\bibitem{l2_41}
S.~P. Norton and R.~A. Wilson.
\newblock A correction to the 41-structure of the Monster, a construction of a
  new maximal subgroup $\psl{2}{41}$ and a new Moonshine phenomenon.
\newblock {\em J.\ London Math.\ Soc.}, 87:943–962, 2013.





\bibitem{pisgit}
A.~Pisani.
\newblock Jupyter Notebook with accompanying code \url{https://github.com/apsmath/MonsterSubgroups}


\bibitem{four_fus}
A.~Pisani and T.~Popiel.
\newblock Conjugacy class fusion from four maximal subgroups of the Monster.
\newblock {\em J.\  Comput.\ Alg.}, 11:100021, 2024.




\bibitem{pgl_pres}
E.~F. Robertson and P.~D. Williams.
\newblock A presentation of $\pgl{2}{p}$ with three deﬁning relations.
\newblock {\em Proc.\  Edinburgh Math.\ Soc.}, 27:145--149,
  1984.








\bibitem{serpent}
M.~Seysen.
\newblock A computer-friendly construction of the monster.
\newblock preprint, 2024.
  {\url{arxiv.org/abs/2002.10921}}

\bibitem{fast_monster}
M.~Seysen.
\newblock A fast implementation of the Monster group.
\newblock {\em J. Comput.\ Alg.}, 9:100012, 2024.


\bibitem{mmgroup}
M.~Seysen.
\newblock The \mmgroup{} Package.
\newblock Online, 2024.
  {\url{https://github.com/Martin-Seysen/mmgroup}}


\bibitem{2local1}
S.~V. Shpectorov and U.~Meierfrankenfeld.
\newblock {\em Maximal 2-local subgroups of the Monster and Baby Monster},
volume 487 of {\em London Math.\ Soc.\ Lect.\ Note Ser.}, pages
  200--246.
\newblock Cambridge University Press, 2024.




  
\bibitem{sl_pres_1}
J.~G. Sunday.
\newblock Presentations of the Groups $\lins{2}{m}$ and $\textup{L}_2 ( m
  )$.
\newblock {\em Canad.\ J.\ of Math.}, 24(6):1129--1131, 1972.




\bibitem{online_atlas}
R.~Wilson, P.~Walsh, J.~Tripp, I.~Suleiman, R.~Parker, S.~Norton, S.~Nickerson,
  S.~Linton, J.~Bray, and R.~Abbott.
\newblock ATLAS of Finite Group Representations - Version 3.
\newblock Online, 2024.
  {\url{http://atlas.math.rwth-aachen.de/}}
 
\bibitem{odd_local}  
R.~A. Wilson.
\newblock The odd-local subgroups of the Monster.
\newblock {\em J.\  Austral.\ Math.\ Soc.}, 44:1--16, 1988.



\bibitem{2e62}
  R.~A.\ Wilson. Maximal Subgroups of ${}^2 \textup{E}_6 ( 2 )$ and its automorphism groups.
  preprint, \url{arxiv.org/abs/1801.08374}, 2018.




\bibitem{survey}
R.~A. Wilson.
\newblock Maximal Subgroups of Sporadic Groups.
\newblock volume 694 of {\em Contemporary Mathematics}, pages 57--72,
  Providence, RI. Amer.\ Math.\ Soc., 2017.


\end{thebibliography}
\end{document}